\documentclass{amsart}

\usepackage{amssymb,amsthm,epsfig,color,xfrac}
                                
\usepackage[colorlinks=true, linkcolor=blue, citecolor=green]{hyperref}
\usepackage{MnSymbol}

\usepackage{pinlabel}

\usepackage[all]{xy}
\newcommand{\cd}[1]{\begin{equation*}{\xymatrix{#1}}\end{equation*}}
\newcommand{\cdlabel}[2]{\begin{equation}\label{#1}{\xymatrix{#2}}\end{equation}}

\newcommand{\mc}[1]{\mathcal{#1}}
\newcommand{\mb}[1]{\mathbf{#1}}
\newcommand{\co}{\colon\thinspace}

\newcommand{\bN}{\mathbb{N}}
\newcommand{\bZ}{\mathbb{Z}}

\def\into {\, {\hookrightarrow}\, }

\newcommand{\diam}{\mathrm{diam}}
\newcommand{\Top}{\mathbf{Top}}
\newcommand{\Grp}{\mathbf{Grp}}

\def\QVH {\mc{QVH}}

\DeclareMathOperator\dotnorm{\dot{\lhd}}
\DeclareMathOperator\dotsub{\dot{<}}

\newtheorem{theorem}{Theorem}[section]
\newtheorem{lemma}[theorem]{Lemma}
\newtheorem{corollary}[theorem]{Corollary}
\newtheorem{proposition}[theorem]{Proposition}

\newtheorem{observation}[theorem]{Observation}
\newtheorem{case}{Case}

\newtheorem{claim}{Claim}[theorem]
\newtheorem*{claim*}{Claim}

\theoremstyle{definition}
\newtheorem{remark}[theorem]{Remark}
\newtheorem{definition}[theorem]{Definition}
\newtheorem{example}[theorem]{Example}
\newtheorem{notation}[theorem]{Notation}

\makeatletter
\newtheorem*{rep@theorem}{\rep@title}
\newcommand{\newreptheorem}[2]{%
\newenvironment{rep#1}[1]{%
 \def\rep@title{#2 \ref{##1}}%
 \begin{rep@theorem}\it}%
 {\end{rep@theorem}}}
\makeatother

\newreptheorem{theorem}{Theorem}
\newreptheorem{lemma}{Lemma}

\title[Malnormal Special Quotient Theorem]{An alternate proof of Wise's Malnormal Special Quotient
  Theorem}

\author{Ian Agol}
\address{
   University of California, Berkeley,
   970 Evans Hall \#3840,
   Berkeley, CA 94720-3840}
\email{ianagol@math.berkeley.edu}
\thanks{The first author is supported by NSF grant DMS-1105738, the Miller Institute, and the Simons Foundation}

\author{Daniel Groves}
\address{Department of Mathematics, Statistics, and Computer Science,
University of Illinois at Chicago,
322 Science and Engineering Offices (M/C 249),
851 S. Morgan St.,
Chicago, IL 60607-7045}
\email{groves@math.uic.edu}
\thanks{The second author is supported by NSF Grant DMS-0953794}

\author{Jason Fox Manning}
\address{Department of Mathematics, 310 Malott Hall, Cornell University, Ithaca, NY 14853}
\email{jfmanning@math.cornell.edu}
\thanks{The third author is supported by NSF Grant DMS-1104703.}

\begin{document}
\begin{abstract}
  We give an alternate proof of Wise's Malnormal Special
  Quotient Theorem (MSQT), avoiding cubical small cancellation theory.  We also show how to deduce Wise's Quasiconvex Hierarchy Theorem from the MSQT and theorems of Hsu--Wise and Haglund--Wise.
\end{abstract}

\maketitle
\setcounter{tocdepth}{1}
\tableofcontents  
\section{Introduction}

The role of {\em subgroup separability} is a central one in geometric group theory and topology.  In particular, as witnessed by Scott's Criterion \cite{Scott78}, it is intimately linked to the problem of promoting immersions to embeddings in finite covers.  After the work of Kahn and Markovic \cite{KM12}, the Virtual Haken Conjecture was reduced to proving separability of certain surface subgroups of $3$--manifold groups.  The key to the first author's proof of this separability in \cite{VH} was the notion of {\em virtually special cube complexes}, as developed by Wise and his collaborators.

In \cite{HW08} Haglund and Wise proved that if $G$ is a hyperbolic group which acts properly and cocompactly on a CAT$(0)$ cube complex then the action is virtually special if and only if all quasi-convex subgroups of $G$ are separable.  This leads to the desirability of finding criteria which show that a group acts in a virtually special (and proper cocompact) way on a CAT$(0)$ cube complex.

The results of Hsu-Wise \cite{HsuWise15} and Haglund-Wise \cite{HaglundWise12} together show that when a hyperbolic group $G$ has a {\em malnormal quasi-convex hierarchy} then it has the desired type of action.  See Section \ref{s:hierarchies} for the definition of hierarchies and Section \ref{s:SCT} for more details, in particular Theorem \ref{MQH} which combines the results of \cite{HsuWise15} and \cite{HaglundWise12}.  In the unpublished work \cite{Wise} Wise proves that hyperbolic groups with {\em quasi-convex hierarchies} also admit such actions, and hence all their quasi-convex subgroups are separable.  Wise's result is much more broadly applicable than Theorem \ref{MQH} because subgroups appearing in the hierarchy in many natural applications will not be malnormal.  See Section \ref{s:QCH} for more details and a proof of Wise's result, which appears as Theorem \ref{t:QVH} in this paper.

Subsequently, the first author \cite{VH} proved that {\em any} hyperbolic group which acts properly and cocompactly on a CAT$(0)$ cube complex acts in a virtually special way, with no further hypotheses.  One of the main purposes of this paper is to explain some of the ingredients (from \cite{Wise}) that go into this proof, so we do not assume the results from \cite{VH} here.  In fact, this paper is intended to provide an alternative account of the results needed for \cite{VH} without relying on any of the results from \cite{Wise}.  It is important to note that Wise proves various results in greater generality than is required for \cite{VH}, and in this paper we do not recover the full strength of the results in \cite{Wise}.  In particular, we work throughout with hyperbolic groups, whereas certain of Wise's results are proved for more general relatively hyperbolic groups.

As we explain in Section \ref{s:QCH} (and is made clear in \cite{Wise}) the key result for proving Wise's Theorem \ref{t:QVH} is the {\em Malnormal Special Quotient Theorem  (MSQT)}, and the main purpose of the current paper is to provide a new proof of this theorem.  In fact, we prove Theorem \ref{msqt+} below, which is a generalization of Wise's MSQT.

Our key innovation in the proof is a new kind of hierarchy for virtually special hyperbolic groups (described in Section \ref{s:vhierarchy}).  It is a consequence of work of Haglund and Wise \cite{HW08} that any virtually special hyperbolic group has a finite-index subgroup with a malnormal quasi-convex hierarchy terminating in the trivial group.  However, the MSQT requires us to `kill' certain subgroups of a malnormal family $\mc{P}$ of quasiconvex `peripheral' subgroups. 
We construct (in Theorem \ref{thm:virtualhierarchy}) a new hierarchy (of a different finite-index subgroup) which is quasiconvex but does not terminate in the trivial group.  The peripheral subgroups of the finite index subgroup are conjugate to finite index subgroups of the elements of $\mc{P}$.
The virtue of this new hierarchy is that its peripheral subgroups are elliptic at every stage of the hierarchy, and furthermore when they intersect an edge group of a splitting in an infinite set they are entirely contained in the edge group.  (See Section \ref{ss:relativehierarchies} and particularly Definition \ref{def:Pelliptic} for more details.)   Therefore (using certain results about Relatively Hyperbolic Dehn Filling -- see Section \ref{s:RHDF}) the hierarchy and the filling are operations which `commute', and we find a hierarchy of the filled group.  Using the results of Haglund-Wise and Hsu-Wise described above, we deduce that the filled group is virtually special, as required.  In Section \ref{sec:combination} we prove a Combination Theorem which (as with most combination theorems) states that certain subgroups are amalgamated free products or HNN extensions, and are also quasi-convex.  This allows us to see that the hierarchy which we build topologically has the expected algebraic structure, which in turn allows us to deduce that the hierarchy is quasi-convex.  We now provide a more detailed and technical outline of the main work in this paper.

\section{Outline}

Special cube complexes are defined in \cite[Definition 3.2]{HW08}.  Note that \cite[Proposition 3.10]{HW08} states that a special complex has a finite $A$-special cover, and many people use the term `special' to mean $A$-special.  Since we are only interested in this property virtually, the distinction is not important.  In fact, we do not need to use the definition of special cube complexes in this paper. Rather, we use the notion as a black box by quoting results, particularly those from \cite{HW08, HsuWise15, HaglundWise12}.
\begin{definition}
  We say that a group $G$ is \emph{special} if it is the fundamental
  group of a compact special cube complex.  (Other authors use the
  term \emph{compact special}.)

  A group is {\em virtually special} if it has a finite-index subgroup which is special.
\end{definition}

\begin{remark}\label{remark:vspecialgroup}
If $G$ is a hyperbolic group which acts properly and cocompactly on a CAT$(0)$
cube complex $\tilde{X}$ then $G$ 
is virtually special if and only if all quasi-convex subgroups of $G$ are separable \cite[Theorems 1.3 and 1.4]{HW08}.  Therefore,
being virtually special is a property of the (hyperbolic) group, not the particular cube complex upon which it acts properly and cocompactly.

As mentioned in the introduction, the Main Theorem from \cite{VH} shows that 
{\em any} hyperbolic group which acts properly and cocompactly on a CAT$(0)$ cube complex is virtually special.  However, in this paper we are giving an alternative account of some of the ingredients of the proof of this theorem, and so we do not assume this result from \cite{VH}.
\end{remark}

\begin{notation}
  Let $A$ be a group.
  The notation $B\dotsub A$ indicates that $B<A$ and $|A:B|<\infty$.
  Similarly, the notation $B\dotnorm A$ indicates $B\lhd A$ and $|A:B|<\infty$.
\end{notation}
\begin{definition}\label{def:dehnfill}
  Let $\mc{P} = \{P_1,\ldots,P_m\}$ be a finite collection of
  subgroups of a group $G$ (called a \emph{peripheral structure} for $G$).  A choice of subgroups (called \emph{filling kernels}) 
  $\{N_i\lhd P_i\}_{i=1}^n$ gives rise to a \emph{(Dehn) filling} of $(G,\mc{P})$:
\[ G \twoheadrightarrow G(N_1,\ldots,N_m) \]
  with kernel $\llangle  \cup_i N_i\rrangle_G$.
  A Dehn filling $G\to G(N_1,\ldots,N_M)$ is called \emph{peripherally
    finite} if $N_i\dotnorm P_i$ for all $i$.
\end{definition}

We want to apply Dehn filling to relatively hyperbolic pairs $(G,\mc{P})$ where $G$ is itself hyperbolic.  To see when we can do this, recall the definition of a malnormal collection of subgroups.
\begin{definition}
  A collection $\mc{P}$ of subgroups of $G$ is said to be \emph{malnormal} (resp. \emph{almost malnormal}) if for any $P,P'\in \mc{P}$, and $g\in G$, either $P^g\cap P'$ is trivial (resp. finite), or $P=P'$ and $g\in P$.
\end{definition}
The following well known characterization of relatively hyperbolic pairs $(G,\mc{P})$ where $G$ is hyperbolic follows from results of Bowditch and Osin.
\begin{theorem} \cite{bowditch:relhyp,osin:relhypbook} \label{thm:when hyp is RH}
Suppose $G$ is hyperbolic and $\mc{P}$ is a finite collection of subgroups of $G$.  The pair $(G,\mc{P})$ is relatively hyperbolic if and only if $\mc{P}$ is an almost malnormal collection of quasiconvex subgroups.
\end{theorem}
\begin{proof}
  The forward direction follows from some much more general facts about parabolic
  subgroups of relatively hyperbolic groups:  If $(G,\mc{P})$ is any
  relatively hyperbolic pair, then the collection $\mc{P}$ is almost
  malnormal by \cite[Proposition 2.36]{osin:relhypbook}.  Moreover the
  elements of $\mc{P}$ are undistorted in $G$ 
  \cite[Lemma 5.4]{osin:relhypbook}.  
  Undistorted subgroups of a hyperbolic group are quasiconvex.

  The other direction is proved by Bowditch \cite[Theorem 7.11]{bowditch:relhyp}.
\end{proof}

Here is our main result:
\begin{theorem}[Main Theorem]\label{msqt+}
  Let $G$ be hyperbolic and virtually special, and suppose $(G,\mc{P})$ is relatively
  hyperbolic.  
  There are subgroups $\{\dot{P}_i\dotnorm P_i\}$ so that
  if $\bar{G}=G(N_1,\ldots,N_m)$ is any
  filling  with $N_i<\dot{P}_i$ and $P_i/N_i$ virtually special and
  hyperbolic 
  for all $i$, then $\bar{G}$ is hyperbolic and virtually special.
\end{theorem}

Since finite groups are virtually special and hyperbolic, we recover the following
theorem of Wise.

\begin{corollary}[The Malnormal Special Quotient Theorem, \cite{Wise}] \label{msqt}
  Let $G$ be hyperbolic and virtually special, and suppose $(G,\mc{P})$ is relatively
  hyperbolic.  
  There are subgroups $\{\dot{P}_i\dotnorm P_i\}$ so that
  if $\bar{G}=G(N_1,\ldots,N_m)$ is any
  peripherally finite filling  with $N_i<\dot{P}_i$
  for all $i$, then $\bar{G}$ is hyperbolic and virtually special.
\end{corollary}

The proof of Theorem \ref{msqt+} relies on the following theorem, which is \cite[Theorem 11.2]{Wise} in Wise's manuscript.   In Section \ref{s:hierarchies} we explain the terminology, and in Section \ref{s:SCT} we recall how the result follows almost immediately from the work of Hsu-Wise \cite{HsuWise15} and Haglund-Wise \cite{HaglundWise12}.
\begin{reptheorem}{MQH}{\em (Malnormal Quasiconvex Hierarchy)}  If $G$ is a hyperbolic group with a malnormal quasiconvex hierarchy terminating in a collection of virtually special groups, then $G$ is virtually special.
\end{reptheorem}
Thus, in order to prove Theorem \ref{msqt+}, it suffices to prove that $\bar{G}$ (under suitable hypotheses on the filling) is a hyperbolic group which (virtually) has a malnormal quasiconvex hierarchy terminating in a collection of virtually special groups.  The key results for proving that this is the case are Theorems \ref{thm:virtualhierarchy} and \ref{thm:hierarchydescends} below.

\begin{notation}[Peripheral structure on a finite index subgroup] \label{subpair}
Given $G'\dotnorm G$, a peripheral structure $\mc{P}$ on $G$ induces a
peripheral structure $\mc{P'}$ on $G'$: For each $P\in \mc{P}$, let
$\mc{E}_0(P) = \{gPg^{-1}\cap G'\mid g\in G\}$, and let $\mc{E}(P)$ be
obtained from $\mc{E}_0(P)$ by choosing one element of each
$G'$--conjugacy class of subgroup.  Note that $\mc{E}(P)$ is finite,
and each element is conjugate to a finite index subgroup of $P$.  Now
let
\[ \mc{P}' = \bigsqcup_{P\in \mc{P}} \mc{E}(P).\]
We summarize all this with the notation
$(G',\mc{P}')\dotnorm (G,\mc{P})$. 

Note that $(G',\mc{P}')$ is relatively hyperbolic if and only if $(G,\mc{P})$ is.
\end{notation}
\begin{remark}
  The collection $\mc{P}'$ just defined can also be recovered from
  topology, if we keep in mind the fundamental group of a subcomplex
  $Y\subset X$ is only defined up to conjugacy in $\pi_1X$.  Suppose
  $G = \pi_1X$ and $\mc{Y}$ is a collection of subcomplexes of $X$
  whose fundamental groups correspond to $\mc{P}$.  Let $X'\to X$ be a
  regular finite-degree cover, so $\pi_1X' = G'\dotnorm G$.  Let
  $\mc{Y}'$ be the collection of preimages of elements of $\mc{Y}$ in $X'$.  Then $\mc{P}'$ is the set of fundamental groups of
  elements of $\mc{Y}'$.
\end{remark}

Our main result is obtained from the following two statements.  (See Section \ref{s:hierarchies} below for definitions of the terminology relating to hierarchies.)
\begin{theorem}[Virtual Relative Hierarchy]\label{thm:virtualhierarchy}
  Given $(G,\mc{P})$ relatively hyperbolic, with $G$ hyperbolic and virtually special, there is some $(G',\mc{P}')\dotnorm (G,\mc{P})$ which has a malnormal quasiconvex fully $\mc{P'}$-elliptic hierarchy terminating in $\mc{P}'$.
\end{theorem}

Theorem \ref{thm:virtualhierarchy} is proved in Section \ref{sec:vhier}.

\begin{theorem}[Hierarchy Filling]\label{thm:hierarchydescends}
  Suppose $(G,\mc{P})$ is relatively hyperbolic with $G$ hyperbolic and having a malnormal quasiconvex fully $\mc{P}$-elliptic hierarchy terminating in $\mc{P}$.  For all sufficiently long fillings,
  \[ (G,\mc{P}) \longrightarrow (\bar{G},\bar{\mc{P}}), \]
with all elements of $\bar{\mc{P}}$ hyperbolic,  the group $\bar{G}$ is hyperbolic and has a malnormal quasiconvex hierarchy terminating in $\bar{\mc{P}}$.
\end{theorem}

Theorem \ref{thm:hierarchydescends} is proved in Section \ref{sec:hierdescends}.

In Section \ref{sec:mainthm} we prove the main theorem from the two theorems above.
 Briefly one argues as follows: Given $(G,\mc{P})$
as in the theorem, one passes to a finite index subgroup
$(G_0,\mc{P}_0)$ satisfying the conclusion of Theorem
\ref{thm:virtualhierarchy}.  This subgroup has a malnormal quasiconvex
hierarchy terminating in peripheral subgroups.  Theorem
\ref{thm:hierarchydescends} says that any sufficiently long filling of
$(G_0,\mc{P}_0)$ still has such a hierarchy.  Now assume a filling (in addition to being sufficiently long)
has the property that $P/N$ is always virtually special, for $P\in
\mc{P}_0$ and $N$ the corresponding filling kernel.  Using the Relatively Hyperbolic Dehn Filling Theorem \ref{thm:RHDF}, these $P/N$ are precisely the peripheral subgroups of the filling.  
Using the
Malnormal Quasiconvex Hierarchy Theorem  \ref{MQH} the
filling is virtually special.  If the filling was chosen appropriately, it
gives a filling of $(G,\mc{P})$ which is therefore also virtually
special.

\subsection{Conventions}
Let $\Delta$ be a geodesic triangle with vertices $\{x,y,z\}$ in a geodesic space $X$.  The triangle determines a metric tripod $T_\Delta$, and a comparison map 
$\phi\co \Delta\to T_\Delta$, isometric on each side of $\Delta$.  The leg of $T_\Delta$ adjacent to $\phi(x)$ has length equal to the \emph{Gromov product} 
\[ (y,z)_x:= \frac{1}{2}\left( d(x,y)+d(x,z)-d(y,z) \right).\]
The triangle $\Delta$ is said to be \emph{$\delta$--thin} if $\phi^{-1}(t)$ has diameter at most $\delta$, for all $t\in T_\Delta$.
For us a space is \emph{$\delta$--hyperbolic} if all the triangles are \emph{$\delta$--thin}.  (For the relationship between this and other definitions of hyperbolicity, see \cite[Chapter III.H]{bridhaef:book}.)

\section{Hierarchies} \label{s:hierarchies}

\subsection{Hierarchies of groups}
We refer to Serre \cite{serre:trees} for more detail on the basic
theory of graphs of groups, their fundamental groups, and the
relationship with actions on trees.  All the underlying graphs of a graph of groups will be finite.  In contrast to Serre, we do not require that the maps
from edge-groups to vertex groups be injective.  Rather, in our applications a graph of groups is found from a graph of spaces via van Kampen's Theorem, and much of our work is to prove that the maps from edge-spaces to vertex-spaces are $\pi_1$-injective.

 Recall that for
Serre, a graph is a pair of sets $(V,E)$, with a fixed-point free
involution $e\mapsto \bar{e}$ on $E$ (exchanging an oriented edge with its
orientation reverse), and a ``terminus'' map
$t\co E\to V$ mapping an edge to the vertex it is oriented toward.
The ``initial'' point of an edge is $i(e) := t(\bar{e})$.

\begin{definition}
  A \emph{graph of groups} $(\Gamma,\mc{G})$ is:
  \begin{enumerate}
  \item A graph $\Gamma = (V,E)$;
  \item An assignment $\mathcal{G}\co V\sqcup E\to \Grp$ of a group to each edge and vertex of $\Gamma$, satisfying $\mathcal{G}(\bar{e})= \mathcal{G}(e)$; and
  \item For each edge $e\in E$, a homomorphism $\phi_e\co \mc{G}(e)\to \mc{G}(t(e))$.
  \end{enumerate}
The graph of groups is \emph{faithful} if all homomorphisms $\phi_e$ are injective.  
\end{definition}

\begin{remark}
  For any graph of (pointed) spaces (see Section \ref{hierarchyofspaces} below), there is a corresponding graph of groups obtained by applying the $\pi_1$ functor to the spaces and maps, and using the same underlying graph.  The groupoid of homotopy classes of paths between basepoints of vertex spaces is isomorphic to the fundamental groupoid of the corresponding graph of groups.  Similarly, the fundamental group of the graph of spaces is isomorphic to the fundamental group of the graph of groups.
\end{remark}

\begin{definition}
  A \emph{graph of groups structure for $G$} is a graph of groups $(\Gamma,\mc{G})$ together with an 
  isomorphism between $G$ and $\pi_1(\Gamma,\mc{G},v)$, the fundamental group of the graph of groups based at some vertex $v$.  The structure is \emph{degenerate} if the graph $\Gamma$ consists of a single point $v$, with $\mc{G}(v) = G$, and the isomorphism is the identity map.
  We'll say the structure is \emph{faithful} if $(\Gamma,\mc{G})$ is faithful.
\end{definition}

\begin{definition}
  A \emph{hierarchy of length $0$} for $G$ is the degenerate graph of groups structure for $G$.
  Let $n>0$.
  A \emph{hierarchy of length $n$} for $G$ is a (finite) graph of groups structure for $G$, together with a hierarchy of length $(n-1)$ for each vertex group.  
\end{definition}
Note that a hierarchy of length $1$ for $G$ is the same thing as a graph of groups structure for $G$.  Longer hierarchies can be thought of as ``multi-level graphs of groups''.
\begin{definition}[Levels of a hierarchy, terminal groups]
Let $G_v$ be a vertex group of this graph of groups structure.  We say that $G_v$ is at \emph{level $1$} of the hierarchy.  The groups at level $1$ of the hierarchy for $G_v$ are at \emph{level $2$} of the hierarchy for $G$, and so on.
If $\mc{H}$ is a hierarchy, we sometimes write ``$H\in \mc{H}$'' to mean that $H$ is at level $i$ of $\mc{H}$ for some $i$.

If a hierarchy has length $n$, the groups at level $n$ are called the \emph{terminal groups} of the hierarchy.
\end{definition}
It is convenient in this definition to allow any or all of the graphs of groups to be degenerate, though of course we are mostly interested in the case when at least some of the graphs of groups at each level are nondegenerate.

\begin{definition}
A {\em faithful} hierarchy is one for which every graph of groups occurring in the hierarchy is faithful.
\end{definition}
In a faithful hierarchy, each vertex or edge group at any level embeds in $G$, and so can be thought of as a subgroup of $G$.  In general $H\in \mc{H}$ only gives a subgroup up to conjugacy in $G$, but we can pin it down if necessary by choosing maximal trees in all the underlying graphs of the hierarchy.

\begin{definition} \hypertarget{terminating}
  Let $\mc{H}$ be a faithful hierarchy of $G$, and let $\mc{P}$ be a
  collection of subgroups of $G$.  If every terminal group of $\mc{H}$
  is conjugate to an element of $\mc{P}$, we say that the hierarchy
  \emph{terminates in $\mc{P}$}.
\end{definition}

\begin{remark}[Warning]
  Even in the faithful case, our terminology differs from Wise's in \cite{RAAGs,Wise}, in two ways.  First, Wise's hierarchies consist of a sequence of \emph{one-edge} splittings, and not general graphs of groups.  Our hierarchies could be converted to ones satisfying this property, at the expense of lengthening them.
Second, Wise usually requires the terminal groups to be trivial (sometimes virtually special), and we don't have any such requirement in general.  
\end{remark}

Sometimes hierarchies have some extra properties which we like.  Here are some of them.
\begin{definition}
  Let $\mc{H}$ be a faithful hierarchy for the group $G$.  Each $H\in \mc{H}$
  comes equipped with a graph of groups structure.  Write $\mc{E}(H)$
  for the set of edge groups of this structure.  We say the hierarchy $\mc{H}$ is
  \begin{enumerate}
  \item \emph{quasiconvex} if $G$ is finitely generated and, for all $H\in \mc{H}$, $K\in \mc{E}(H)$, $K$ is quasi-isometrically embedded in $H$; and
  \item\label{almost malnormal} \emph{(almost) malnormal} if, for all $H\in \mc{H}$, $K\in \mc{E}(H)$, $K$ is (almost) malnormal in $G$.  
  \end{enumerate}
\end{definition}
\begin{remark}[Different kinds of malnormal hierarchy]
  There are a number of reasonable choices for the definition of a malnormal hierarchy.  A stronger version might require $\mc{E}(H)$ to be a malnormal collection in $G$ (or just in $H$).  A weaker (and more common) version just requires the elements of $\mc{E}(H)$ to be malnormal in $H$ and not necessarily in $G$.  
\end{remark}

\subsection{Homotopy hierarchies of spaces}\label{hierarchyofspaces}
Hierarchies of spaces are multi-level graphs of spaces, so we start by
recalling what is meant by a graph of spaces.

\begin{definition}
  A \emph{graph of spaces} $(\Gamma, \mc{X})$ consists of the following data:
  \begin{enumerate}
  \item A graph $\Gamma=(V,E)$;
  \item an assignment of path-connected spaces $\mc{X}\co V\sqcup E\to \Top$ satisfying $\mc{X}(\bar{e})= \mc{X}(e)$;
  \item for each edge $e$ a continuous map $\psi_e\co \mc{X}(e)\to \mc{X}(t(e))$.
  \end{enumerate}
  For each $v\in V$, let $E_v = \{e\in E\mid t(e) = v\}$ be the collection of edges incident on $v$.  We combine all the maps $\psi_e$ for edges incident on $v$ to get a map
  \[ \Psi_v \co \bigsqcup_{e\in E_v} \mc{X}(e) \to \mc{X}(v). \]
  The \emph{semistar} 
   $S(v)$ is defined to be the mapping cylinder of $\Psi_v$.  The \emph{realization} $R(\Gamma,\mc{X})$ is then defined to be the union of the semistars, glued together using the identifications $\mc{X}(e)= \mc{X}(\bar{e})$.  We'll sometimes abuse language and refer to this realization as a graph of spaces. 
\end{definition}

\begin{definition}
  A \emph{graph of spaces structure} on $X$ is a homotopy equivalence $h\co R(\Gamma,\mc{X})\to X$ for some graph of spaces $(\Gamma,\mc{X})$.  The structure is \emph{degenerate} if $\Gamma$ consists of a single point labeled by $X$, and the homotopy equivalence is the identity.  
\end{definition}
\begin{remark}\label{remark:graphsandspaces}
  A graph of spaces $(\Gamma,\mc{X})$ gives rise to a graph of groups
  $(\Gamma,\mc{G})$ in a straightforward way.  Indeed, for each vertex
  or edge space $\mc{X}(x)$, one may choose a basepoint $b_x$.  The
  edge or vertex group $\mc{G}(x)$ is defined to be
  $\pi_1(\mc{X}(x),b_x)$.  The maps $\psi_e$ need not be basepoint
  preserving, so we also make a choice, for each edge $e$, of a path
  $\sigma_e$ from $\psi_e(b_e)$ to $b_{t(e)}$.  Such a path gives an
  identification of $\mc{G}(v)$ with $\pi_1(\mc{X}(v),\psi_e(b_e))$.
  Using this identification, we can define $\phi_e\co \mc{G}(e)\to
  \mc{G}(t(e))$ to be $(\psi_e)_*$.  

  The fundamental groupoid of this graph of groups \cite{higgins76}
  is isomorphic to
  the groupoid of homotopy classes of paths in $R(\Gamma,\mc{X})$ with
  endpoints in $\{b_v\mid v\mbox{ a vertex of }\Gamma\}$.

  Conversely, given any graph of groups we may build a corresponding
  graph of basepointed spaces in a straightforward way, realizing each
  edge and vertex space by a $K(\pi,1)$, and setting the maps $\psi_e$
  to be continuous pointed maps inducing the homomorphisms $\phi_e$.
\end{remark}

\begin{definition}
  A \emph{hierarchy of length $0$} for a space $X$ is the degenerate graph of spaces structure on $X$.

  A \emph{hierarchy of length $n$} for $X$ is a graph of spaces
  structure on $X$, together with a hierarchy of length $(n-1)$
  on each vertex space.

  The \emph{levels} of the structure are the collections of vertex spaces at each stage.  That is, the vertex spaces at level one are the vertex spaces of the graph of spaces structure on $X$; those at level $2$ are the vertex spaces of the spaces at level $1$, and so on.

  The vertex spaces at level $n$ of a length $n$ hierarchy are called the \emph{terminal spaces}.
\end{definition}

At each level of a hierarchy, there is a collection of homotopy
equivalences connecting the graphs of spaces at that level with the
vertex spaces of the previous level.  In the following two examples,
these homotopy equivalences are actually homeomorphisms.  Note however
that in our proof of the Virtual Hierarchy Theorem \ref{thm:virtualhierarchy} we need to consider hierarchies in which the homotopy
equivalences are \emph{not} homeomorphisms.
\begin{example}[Special Cube Complex] \label{ex:Hierarchy for cube}
  The hyperplanes of a (compact) special cube complex $X$ induce a
  hierarchy of spaces which is particularly simple.  Order the
  hyperplanes $H_1,\ldots H_n$.  Each hyperplane $H_i$ is two-sided, so the (open) cubes which intersect it give a product neighborhood $N(H_i)\cong H_i\times (-1,1)$.
  Each hyperplane $H_i$ is the edge space for a graph of spaces structure with underlying graph either an edge (if $H_i$ is separating) or a loop (otherwise).  The (one or two) components of $X\smallsetminus N(H_i)$ are the vertex spaces.
  The order allows us to use these decompositions to define a hierarchy.

  Level $0$ is the cube complex $X$; level $1$ consists of the
  components of $X\smallsetminus N(H_1)$; level $k$ consists of the
  components of $X\smallsetminus (\bigcup_{i=1}^k N(H_i))$.  The edge
  spaces at level $k$ are the components of $H_k\smallsetminus(\bigcup_{i=1}^{k-1} N(H_i))$.  
  The terminal spaces are points.
  \end{example}
\begin{example}[Haken $3$--manifold]
  Each surface in the Haken hierarchy has a product neighborhood, so
  we can regard this hierarchy as a hierarchy of spaces in an obvious
  way.  The terminal spaces are balls.
\end{example}

\begin{remark}
  Using Remark \ref{remark:graphsandspaces}, each graph of spaces
  gives rise to a graph of groups, via some choices of basepoints and
  connecting arcs.  A hierarchy of spaces therefore gives rise to a
  hierarchy of groups in a straightforward way.

There is also a general construction of a hierarchy of spaces from any
hierarchy of groups, following the last paragraph of Remark
\ref{remark:graphsandspaces}.  In this construction the ability to use
homotopy equivalences rather than homeomorphisms is very convenient.
\end{remark}

\subsection{Induced Hierarchies on subgroups}
The following result follows easily by considering the action of $H<G$ on the Bass-Serre tree corresponding to a graph of groups structure $(\Gamma,\mc{G})$ for $G$.

\begin{proposition}\label{inducedGoG}\cite{bass:covering}
  Suppose that $G$ has a graph of groups structure $(\Gamma, \mc{G})$,
  and that $H\dotnorm G$.  Then $H$ has an induced graph of groups
  structure $(\tilde{\Gamma},\mc{H})$ so that
  \begin{enumerate}
  \item Each vertex group of $(\tilde{\Gamma},\mc{H})$ is equal to $(K^g\cap H) \dotnorm K^g$ for some vertex group $K$ of $(\Gamma,\mc{G})$ and some $g \in G$.
  \item Each edge group of $(\tilde{\Gamma},\mc{H})$ is equal to $(K^g\cap H) \dotnorm K^g$ for some edge group $K$ of $(\Gamma,\mc{G})$ and some $g \in G$.
  \end{enumerate}
\end{proposition}

Proposition \ref{inducedGoG} has the following corollary, by induction on $n$:
\begin{corollary}\label{inducedhierarchy}
  Suppose that $G$ has a hierarchy $\mc{H}$ of length $n$, and that $H\dotnorm G$.  Then the hierarchy on $G$ induces a hierarchy $\mc{H}'$ of length $n$ on $H$.  Moreover, for each $i\in \{0,\ldots,n\}$,
  \begin{enumerate}
  \item Each vertex group at level $i$ of $\mc{H}'$ is equal to $(K^g\cap H)\dotnorm K^g$ for some vertex group $K$ at level $i$ of $\mc{H}$ and some $g\in G$.
  \item Each edge group at level $i$ of $\mc{H}'$ is equal to $(K^g\cap H)\dotnorm K^g$ for some edge group $K$ at level $i$ of $\mc{H}$ and some $g\in G$.
  \end{enumerate}
\end{corollary}

\subsection{Quasiconvex hierarchies of special groups}

\subsubsection{Quasiconvexity}

The following lemma is easy, thinking of the Cayley graph of $G$ as coarsely having the structure of a tree of spaces.
\begin{lemma}
  Suppose that $G$ has a graph of groups structure $(\Gamma,\mc{G})$.
  If the edge groups of $(\Gamma,\mc{G})$ are quasiisometrically
  embedded in $G$, then so are the vertex groups.
\end{lemma}

\begin{corollary}
  If $\mc{H}$ is a quasiconvex hierarchy of $G$, then all the edge and
  vertex groups of $\mc{H}$ are quasiisometrically embedded in $G$.
\end{corollary}

The following lemma is straightforward.
\begin{lemma}\label{inducedquasiconvex}
  If $\mc{H}$ is a quasiconvex hierarchy of $G$, and $G_0\dotnorm G$, then the induced hierarchy $\mc{H}_0$ on $G_0$ is a quasiconvex hierarchy.
\end{lemma}

\subsubsection{Malnormality}
Recall that a subgroup $H<G$ is \emph{malnormal} if $H\cap g H g^{-1}$ is trivial for all $g\notin H$, and is \emph{almost malnormal} if $H\cap g H g^{-1}$ is finite for all $g\notin H$.
The failure of (almost) malnormality of a subgroup is measured by the height.
\begin{definition}[Height]
  Let $H<G$.  The \emph{height} of $H$ in $G$ is the largest number $n\geq 0$ so that there are $n$ distinct cosets $\{g_1H,\ldots,g_n H\}$, so that the intersection of conjugates
  $\bigcap_i g_iH g_i^{-1}$ is infinite.  Thus finite groups have height $0$, infinite almost malnormal subgroups have height $1$, and so on.
\end{definition}
The following lemma is straightforward and left to the reader.
\begin{lemma} \label{heightnotup}
  Suppose $H<G$, and $G_0\dotsub G$.  Let $H_0= H\cap G_0$.  If $H$ is malnormal in $G$, then 
  $H_0$ is malnormal in $G_0$.  More generally, the height of $H_0$ in $G_0$ is at most the height of $H$ in $G$.
\end{lemma}

\begin{definition}
  Say that $H<G$ is \emph{virtually malnormal} if there is a finite
  index subgroup $G_0<G$ so that if $H_0=H\cap G_0$, then $H_0$ is
  malnormal in $G_0$.
\end{definition}

\begin{definition}
  The subgroup $H<G$ is \emph{separable in $G$} (or just \emph{separable} when $G$ is understood) if 
  \[H = \bigcap \{G_0\mid H<G_0\dotsub G\}.\]
\end{definition}

\begin{proposition}\cite{HW08,hruskawise:packing} \label{p:qc+sepimpliesvm}
  If $G$ is hyperbolic and virtually torsion-free, and $H<G$ is
  quasiconvex and separable, then $H$ is virtually malnormal.
\end{proposition}
\begin{proof}
  Hruska and
  Wise \cite[Theorem 9.3]{hruskawise:packing} implies that a separable quasiconvex
  subgroup of a hyperbolic group is virtually \emph{almost} malnormal.
  Indeed they show that
  there is some $G_0\dotsub G$  containing $H$ so that $H$
  is almost malnormal in $G_0$.  By hypothesis, $G$ is virtually torsion-free. Let $G_1\dotsub G$ be a torsion-free
  subgroup, and let $G_2 = G_0\cap G_1$.  We then have 
  $H_2 = H\cap G_2$ malnormal in $G_2$, using Lemma \ref{heightnotup}.
\end{proof}

\begin{theorem}\cite[Theorem 1.3]{HW08} \label{qcthensep}
  Let $G$ be hyperbolic and virtually special, and let $H<G$ be quasiconvex.  Then $H$ is separable in $G$.
\end{theorem}

\begin{corollary}\label{qcimpliesvm}
If $G$ is hyperbolic and virtually special, and $H < G$ is quasiconvex, then $H$ is virtually malnormal.
\end{corollary}
\begin{proof}
  Since $G$ is virtually special it is linear, in particular virtually
  torsion-free.  The subgroup $H$ is separable by Theorem
  \ref{qcthensep}.  We can therefore apply Proposition
  \ref{p:qc+sepimpliesvm}.
\end{proof}

\begin{theorem}\label{virtually malnormal hierarchy}
  If $\mc{H}$ is a quasiconvex hierarchy of a hyperbolic virtually special group
  $G$, then for some $G_0\dotnorm G$, the induced hierarchy $\mc{H}_0$ of $G_0$
  is a malnormal hierarchy.
\end{theorem}
\begin{proof}
  The finitely many edge groups of $\mc{H}$ are all virtually
  malnormal by Proposition \ref{qcimpliesvm}.  Lemma \ref{heightnotup} implies malnormality is
  preserved by passing to further finite index subgroups.  We may therefore
  find $G_0\dotnorm G$ so that for every edge group $H$ of $\mc{H}$,
  the intersection $H\cap G_0$ is malnormal in $G_0$.
  By Corollary \ref{inducedhierarchy}, every edge group of the induced hierarchy is of
  the form $H^g\cap G_0$ for some such edge group $H$, and some $g\in
  G$.  Conjugation by $g\in G$ gives an automorphism of $G_0$, so all
  these edge groups are malnormal in $G_0$.  In particular, any edge
  group at level $i$ is malnormal in the vertex group at level $i-1$
  which contains it, so the hierarchy $\mc{H}_0$ is a malnormal
  hierarchy.
\end{proof}

\subsection{Relative hierarchies} \label{ss:relativehierarchies}
In this paper we're particularly interested in hierarchies relative to a family of peripheral subgroups.
\begin{definition}[$\mc{P}$--elliptic hierarchy] \label{def:Pelliptic}
  Let $G$ be a group with a hierarchy $\mc{H}$, and let $\mc{P}$ be a
  family of subgroups of $G$.  Recall that any vertex group $H\in \mc{H}$ 
  comes equipped with some graph of groups structure $H\cong\pi_1(\Gamma,\mc{G},v)$. 
  We say that $\mc{H}$ is
  \emph{$\mc{P}$--elliptic} if it satisfies the following condition:

  Suppose that $H\in \mc{H}$, 
  and suppose that $P'$ is
  any $G$--conjugate of $P\in \mc{P}$ satisfying $P'\subseteq H$.  Then
  $P'$ is $H$--conjugate into a vertex group of $(\Gamma,\mc{G})$.
  
  A $\mc{P}$-elliptic hierarchy $\mc{H}$ is {\em fully ${\mc{P}}$-elliptic} if whenever $P'$ is any $G$-conjugate of $P \in \mc{P}$ and whenever $E$ is an edge group in $\mc{H}$ then either $P' \cap E$ is finite or else $P' < E$.
\end{definition}
In other words, a hierarchy is fully $\mc{P}$--elliptic if the elements of
$\mc{P}$ are never cut up by the splittings in the hierarchy.  (We think of `cutting' along edge groups, which is the algebraic consequence of cutting along subspaces of a topological space.)

Recall the notation $(G_0,\mc{P}_0)\dotnorm(G,\mc{P})$, which means that $G_0\dotnorm G$, and that the elements of $\mc{P}_0$ are representatives of all the $G_0$--conjugacy classes of subgroup of the form $P^g\cap G_0$ where $P\in \mc{P}$ and $g\in G$.  The following result is straightforward.

\begin{lemma}\label{stillPelliptic}
  Suppose that $(G_0,\mc{P}_0)\dotnorm (G,\mc{P})$ and that $G$ has a fully $\mc{P}$--elliptic hierarchy $\mc{H}$.  Then the induced hierarchy $\mc{H}_0$ of $G_0$ is fully $\mc{P}_0$--elliptic.
\end{lemma}

The results in this section can be easily assembled to give:
\begin{theorem} \label{t:vmalnormal}
  Suppose that $G$ is hyperbolic and virtually special and that $(G,\mc{P})$ is relatively hyperbolic. 
  Suppose $G$ is given a quasiconvex fully $\mc{P}$--elliptic hierarchy $\mc{H}$ \hyperlink{terminating}{terminating} in $\mc{P}$. 
Then for some $(G_0,\mc{P}_0)\dotnorm (G,\mc{P})$, and any $(G',\mc{P}') \dotnorm (G,\mc{P})$ so that $G' \le G_0$, the induced hierarchy $\mc{H}'$ is a quasiconvex malnormal fully $\mc{P}'$--elliptic hierarchy \hyperlink{terminating}{terminating} in $\mc{P}'$.
\end{theorem}
\begin{proof}
  By Theorem \ref{virtually malnormal hierarchy} there is a
  $(G_0,\mc{P}_0)\dotnorm (G,\mc{P})$ so that the induced hierarchy
  $\mc{H}_0$ is malnormal.  The induced hierarchy $\mc{H}'$ on $(G',\mc{P}')$ remains quasiconvex and malnormal by Lemmas \ref{inducedquasiconvex} and \ref{heightnotup}.  Lemma \ref{stillPelliptic} shows that the hierarchy is fully $\mc{P'}$-elliptic. Finally, the
  ``moreover'' part of Corollary \ref{inducedhierarchy} shows that the
  terminal groups of $\mc{H}'$ are a subset of conjugates of $\mc{P}'$.  However, since the elements of $\mc{P}'$ are always elliptic in all graphs of groups decompositions, all elements of $\mc{P}'$ must appear as conjugates of terminal vertex groups of $\mc{H}'$.
\end{proof}

The results in this section (particularly Theorem \ref{t:vmalnormal} above) are used in the proof of Theorem \ref{thm:virtualhierarchy} in Section \ref{s:vhierarchy} below.

\section{A metric combination theorem}\label{sec:combination}
In this section we prove a combination theorem in the setting of $\delta$--hyperbolic CAT$(0)$ spaces which is analogous to the Baker--Cooper combination theorem in hyperbolic manifolds \cite{BC08}.
This combination theorem is Theorem \ref{combotheorem} below.  It is worth noting that this theorem is about CAT$(0)$ spaces in general, not just about cube complexes.  Our application of Theorem \ref{combotheorem} is to show that the hierarchy of spaces that we build in the proof of Theorem \ref{thm:virtualhierarchy} in Section \ref{s:vhierarchy} is faithful and quasiconvex.

\begin{definition}
Let $R \geq 0$.
Suppose that $A\subseteq Y$ is connected.  Let $\pi_A \co Y^A\to Y$ be the cover corresponding to the image of $\pi_1(A)$ in $\pi_1(Y)$.  There is a canonical lift $L \co A \to Y^A$.  Let $\tilde{N}_R(A)$ be the $R$-neighborhood of $L(A)$ in $Y^A$.

We say that $A$ is {\em $R$-embedded in $Y$} if $\pi_A$ is injective on $\tilde{N}_R(A)$, and call the image $\pi_A(\tilde{N}_R(A))$ a \emph{tubular neighborhood} of $A$.

If $A$ is not connected, but the components of $A$ are $R$--embedded with disjoint tubular neighborhoods, we say that $A$ is \emph{$R$--embedded in $Y$}.
\end{definition}

\begin{definition}[Elevation]
  Let $W$ be connected, and let $\phi\co W\to Z$ be some map.
  If $\pi\co \hat{Z}\to Z$ is a cover, then $\phi$ may not lift to $\hat{Z}$ but there is some (minimal) cover $\pi_W\co \hat{W}\to W$ so that the composition $\phi\circ \pi_W$ lifts to a map $\hat{\phi}\co \hat{W}\to \hat{Z}$.  Such a lift (or its image)
is called an \emph{elevation} of $W$ to $\hat{Z}$.  If $\phi$ is an inclusion map, then the elevations of $W$ are just the connected components of $\pi^{-1}(W)$.

  Most of the time, we are interested in the images of elevations, and not the precise maps.  In those cases, we say two elevations are \emph{distinct} if they have different images.
\end{definition}

\begin{definition}
Suppose that $f \co {\mathbb N} \to {\mathbb N}$ is an affine function $f(n)=K n + C$.
We say that $\mc{B} = \{ B_i \}$ forms an {\em $f$-separated family of sub-complexes of $Y$} if for any two distinct elevations $U_1$ and $U_2$ of elements of $\mc{B}$ to $\tilde{Y}$ (the universal cover) and any $D \ge 0$ we have
\[	\mathrm{Diam} \left( N_D(U_1) \cap N_D(U_2) \right) \le f(D)	.	\]
\end{definition}
Note that if there are at least two distinct elevations of elements of
$\mc{B}$, then any function $f$ as in the definition must have $K\geq 1$.

The key example of an $f$--separated family is given by the following
proposition: 
\begin{proposition}\label{malnormalseparated} 
  Let $(G,\mc{H})$ be a relatively hyperbolic pair with $G$
  hyperbolic.  Suppose that $Y$ is a space with $\pi_1Y=G$ and that for each $H \in \mc{H}$ there is a
  path-connected subspace $B_H\subseteq Y$ with $\pi_1(B_H)$ 
  conjugate to $H$.  Then $\mc{B} = \{B_H \}$ is $f$--separated for
  some affine function $f$.
\end{proposition}
\begin{proof}
  Fix a Cayley graph $\Gamma$ for $G$.  This graph is
  $\delta$--hyperbolic for some $\delta\geq 0$, and the cosets $gH$
  for $H\in \mc{H}$ are uniformly $\lambda$--quasiconvex in $\Gamma$
  for some $\lambda>0$.  We may assume $\delta$ and $\lambda$ are both
  integers.

  We first prove the analogous statement for cosets of elements of
  $\mc{H}$ in $G$.  Namely, we find a number $K$ so
  that for any $H_1,H_2$ in $\mc{H}$ (not necessarily distinct), any two distinct cosets $g_1H_1$ and $g_2H_2$, and any $t\in \bN$,
  \begin{equation}\label{groupineq}
    \mathrm{diam}\left(N_t(g_1H_1)\cap N_t(g_2 H_2)\right) \leq Kt + K.
  \end{equation}
  There are finitely many pairs $H_1,H_2$ in $\mc{H}$ so it suffices
  to fix (not necessarily distinct) $H_1$ and $H_2$ and find a $K$
  which works for that pair.  Also note that it suffices to verify the
  inequality \eqref{groupineq} for $g_1 = 1$.

  To find $K$, we first observe that there are
  only finitely many double cosets $H_1 g H_2$ which intersect the
  ball of radius $R := 2\lambda + 2\delta$ about $1$.
  Let $\mc{C} = \{ c_1, \ldots , c_n \} \subset B_R(1)$ be a set of representatives of these cosets.  In case $H_1 = H_2$, we omit the representative for the double coset $H_1 = H_1H_1$ from $\mc{C}$.  
  
\begin{claim} \label{bounded diam}
For each $c \in \mc{C}$, the diameter $\diam \left( N_R(H_1) \cap N_R(cH_2) \right)$ is finite.
\end{claim}
\begin{proof}[Proof of Claim \ref{bounded diam}]
If the diameter is infinite, then the limit sets of $H_1$ and $H_2^c$ would intersect.  By 
\cite[Lemmas 2.6 and 2.7]{gmrs} there would be an infinite order element in $H_1 \cap H_2^c$, contradicting malnormality of $\mc{H}$.
\end{proof}

Let
  \[ M = \max_{c\in \mc{C}}\mathrm{diam}\left(N_R(H_1)\cap
  N_R(cH_2)\right)+2\lambda+1. \]
  
  \begin{claim} \label{claim:separated}
    Unless $H_1=H_2$ and $g\in H_1$, we have $\mathrm{diam}(N_t(H_1)\cap N_t(gH_2))\leq 6t + M$.
  \end{claim}

  \begin{proof}[Proof of Claim \ref{claim:separated}]
  Suppose not.  Then there are elements $x, y \in N_t(H_1)\cap
  N_t(gH_2)$ with $d(x,y)>6t + M$.  There are therefore elements
  $a_x,a_y\in H_1$ and $b_x,b_y\in H_2$ with 
  \[\max\{d(a_x,x),
  d(gb_x,x),d(a_y,y),d(gb_y,y)\}\leq t.\]  
  Consider a pair of $\delta$--thin triangles with vertex sets $\{a_x,gb_y,gb_x\}$ and $\{a_x,gb_y,a_y\}$. (See Figure \ref{fig:quad}.)  
  \begin{figure}[htbp]
    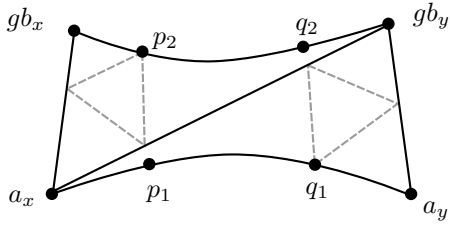
    \caption{The cosets $H_1$ and $gH_2$ must come close.}
    \label{fig:quad}
  \end{figure}
  Note that $d(a_x,gb_y)\geq d(x,y)-2t>M+4t$.  Moreover the Gromov products $(gb_x,gb_y)_{a_x}$ and $(a_x,a_y)_{gb_y}$ are each at most $2t$.  It follows that there are subsegments $[p_1,q_1]\subseteq [a_x,a_y]$ and $[p_2,q_2]\subseteq [gb_x,gb_y]$
of length at least $M$, and whose endpoints satisfy $d(p_1,p_2)\leq 2\delta$ and $d(q_1,q_2)\leq 2\delta$.  The $\lambda$--quasiconvexity of $H_1$ and $H_2$ then implies that there are
  $a_x',a_y'\in H_1$ and $b_x',b_y'\in H_2$ so that 
  \begin{equation}\label{far}
    \min\{d(a_x',a_y'),d(gb_x',gb_y')\}\geq  M-2\lambda,
  \end{equation}
  but 
  \begin{equation}\label{close}
  \max\{d(a_x',gb_x'),d(a_y',gb_y')\}\leq 2\lambda+2\delta = R.
  \end{equation}
  This gives a contradiction:  The inequality \eqref{close} implies 
  that
  $H_1gH_2$ intersects $B_R(1)$ nontrivially, so we must have $\diam(N_R(H_1)\cap N_R(gH_2))\leq M-2\lambda-1$.  On the other hand \eqref{far} implies that $\diam(N_R(H_1)\cap N_R(gH_2))\geq M-2\lambda$.
  \end{proof}

  We can therefore take 
  $K(H_1,H_2) = \max\{6,M\}$, and let $K = \max\{K(H,H')\mid H,H'\in \mc{H}\}$.

  We now derive the Proposition from the group-theoretic statement.
  Let $\tilde{Y}\to Y$ be the universal cover, and for each $i$ let
  $E_i$ be the elevation of $B_i$ to $\tilde{Y}$ which is preserved by
  $H_i$.  For some $\mu\geq 1, \epsilon\geq 0$, we can find
  equivariant maps $\phi\co \tilde{Y}\to \Gamma$ and $\psi\co
  \Gamma\to \tilde{Y}$, so that
  \begin{enumerate}
  \item $\phi$ and $\psi$ are $(\mu,\epsilon)$--quasi-isometries,
  \item $\phi$ and $\psi$ are $\epsilon$--quasi-inverses, and
  \item for each $H_i\in \mc{H}$ and each $g\in G$, the Hausdorff
    distances $d_{\mathrm{Haus}}(\phi(g B_i),gH_i)$ and
    $d_{\mathrm{Haus}}(\psi(g H_i),g B_i)$ are bounded above by $\epsilon$.
  \end{enumerate}
  Now let $t>0$, and choose elevations $U_1$ and $U_2$ of elements
  $B_{i_1}$ and $B_{i_2}\in \mc{B}$.  We have $U_1 = g_1 E_{i_1}$ and $U_2 = g_2
  E_{i_2}$.  We have $\phi(N_t(U_j))\subseteq N_{\mu t + 2\epsilon}(g_j
  H_{i_j})$ for $j=1,2$.  Thus $\phi(N_t(U_1)\cap N_t(U_2))\subseteq
  N_{\mu t + 2\epsilon}(g_1 H_1)\cap N_{\mu t + 2\epsilon}(g_2 H_2)$
  has diameter at most $K(\mu t + 2\epsilon) + K$.  Thus 
  $\psi\circ\phi(N_t(U_1)\cap N_t(U_2))$ has diameter at most 
  $\mu(K(\mu t + 2\epsilon) + K)+\epsilon$.  Since $\phi$ and $\psi$
  are $\epsilon$--quasi-inverses, we get the affine bound
  \[ \mathrm{diam}(N_t(U_1)\cap N_t(U_2))\leq \mu(K(\mu t + 2\epsilon)
  + K)+2\epsilon,\] as required.
\end{proof}

\begin{theorem} [Combination Theorem] \label{combotheorem}
Let $\delta>0$ and let $f\co \bN\to \bN$ be a nondecreasing affine function.
Suppose that $Y$ is a compact NPC space with $\delta$--hyperbolic
universal cover $\tilde{Y}$.  For $\epsilon_B$ large enough (in terms of $\delta$), and $\epsilon_A$ large enough (in terms of $\delta$, $f$,
and $\epsilon_B$), the following combination theorem holds.

Suppose
\begin{enumerate}
\item  $\mc{A} = \sqcup A_i$ and $\mc{B} = \sqcup B_j$ are embedded locally convex subsets (where the $A_i$ are the connected components of $\mc{A}$ and the $B_j$ are the connected components of $\mc{B}$);
\item $\mc{B}$ forms an $f$-separated family;
\item $\mc{A}$ is $\epsilon_A$-embedded in $Y$;
\item Each $B_j \subset \mc{B}$ is $\epsilon_B$-embedded in $Y$; and
\item $\Theta$ is a connected component of $\mc{A}\cup \mc{B}$.
\end{enumerate}
Then $\Theta$ is $\pi_1$-injective with $\lambda$-quasi-convex elevation to $\tilde{Y}$, where $\lambda$ is a function of $\delta$ only.
\end{theorem}

In applications, the components of $\mc{B}$ are thought of as ``peripheral'' subcomplexes.  
With slightly stronger assumptions on how $\mc{A}$ and $\mc{B}$ interact, we are able to rule out ``accidental parabolics''.  
\begin{definition}\label{def:accident}
  Let $\mc{B}$ be an embedded locally convex subset of the
  nonpositively curved space $Y$, and let $\Xi\subset Y$.  We say that
  $\Xi$ has an \emph{accidental $\mc{B}$--loop} if there is a
  homotopically essential loop $\xi\subset \Xi$ so that 
  \begin{enumerate}
  \item\label{inB} $\xi$ is (freely) homotopic to a geodesic loop in $\mc{B}$, and
  \item\label{notinXi} no positive power of $\xi$ is homotopic in $\Xi$ to a geodesic loop in $\mc{B}$.
  \end{enumerate}
\end{definition}
This notion is stable under finite covers:
\begin{lemma}\label{lem:liftedaccidents}
  Let $\mc{B}$, $\Xi$, $Y$ be as in Definition \ref{def:accident}, and let $\pi\co \tilde{Y}\to
  Y$ be a finite cover.  If $\tilde{\mc{B}}= \pi^{-1}(\mc{B})$ and
  $\Xi'$ is an elevation of $\Xi$ to $\tilde{Y}$ then $\Xi$ has an
  accidental $\mc{B}$--loop if and only if $\Xi'$ has an accidental
  $\tilde{\mc{B}}$--loop.  
\end{lemma}
\begin{proof}
  Let $\gamma$ be an accidental $\mc{B}$--loop in $\Xi$, and let $\gamma'$ be an elevation to $\Xi'$.  By condition \eqref{inB} of Definition \ref{def:accident}, $\gamma$ is freely homotopic to a geodesic loop in $\mc{B}$.  We can lift the homotopy to get that $\gamma'$ is freely homotopic to a geodesic loop in $\tilde{\mc B}$.  If $\gamma'$ were freely homotopic to a loop in $\tilde{\mc B}\cap \Xi'$, then we could project the homotopy to get a homotopy of $\gamma^n$ into $\mc{B}\cap \Xi$ for some $n\geq 1$, contradicting condition \eqref{notinXi} for $\xi$.

  Conversely, if $\gamma$ is an accidental $\tilde{\mc B}$--loop in $\Xi'$, we can show that $\bar\gamma =\pi(\gamma)$ is an accidental $\mc{B}$--loop in $\Xi$.  Projecting the homotopy of $\gamma$ into $\tilde{\mc B}$ gives a homotopy of $\bar\gamma$ into $\mc{B}$.  Moreover, if some positive power of $\bar\gamma$ were homotopic in $\Xi$ to a loop in $\mc{B}$ we could lift that homotopy to show some positive power of $\gamma$ was homotopic in $\Xi'$ to a loop in $\tilde{\mc B}$.  
\end{proof}

We'll also prove:
\begin{proposition}\label{noaccidents}
  Let $Y$, $\mc{A}$, and $\mc{B}$ satisfy the assumptions of the Combination
  Theorem \ref{combotheorem}.  Suppose moreover that $\mc{A}$ has no
  accidental $\mc{B}$--loops.  If $\Theta$ is a component of
  $\mc{A}\cup\mc{B}$, then $\Theta$ has no accidental $\mc{B}$--loops.
\end{proposition}

The remainder of this section is devoted to proving Theorem
\ref{combotheorem} and Proposition \ref{noaccidents}.

\subsection{Broken geodesics}
The following Lemma can be proven using essentially the same argument as \cite[III.H.1.13]{bridhaef:book}.
\begin{lemma}\label{broken}
  Let $l\geq 0$.
  Let $c=c_1\cdots c_n$ be a path from $p$ to $q$ in a
  $\delta$--hyperbolic geodesic space, satisfying
  \begin{enumerate}
  \item For each $i$, the subpath $c_i=[p_i,p_{i+1}]$ is geodesic;
  \item for each $i\neq 1$, the Gromov product
    $(p_{i-1},p_{i+1})_{p_i}\leq l$;
  \item for each $i\notin\{1,n\}$, the length of $c_i$ is strictly bigger than $2l+8\delta$.
  \end{enumerate}
  Then if $\gamma$ is any geodesic from $p$ to $q$, the Hausdorff
  distance between $\gamma$ and $c$ is at most $l+5\delta$.
\end{lemma}

\subsection{Proof of Theorem \ref{combotheorem}}

We'll show the following assumptions on $\epsilon_A,\epsilon_B$ are sufficient.

\begin{align} \label{epsilons}
 \epsilon_B & > 50\delta \\
\label{epsilons2} \epsilon_A & >2 \max\left\{ f(\epsilon_B+2 \delta)+\delta +\epsilon_B \mbox{, }2\epsilon_B \right\}
\end{align}

Recall that given $S\subseteq Y$ connected, and $\pi\co Y'\to Y$ any cover, an \emph{elevation of $S$ to $Y'$} is a component of $\pi^{-1}(S)$.  We're particularly interested in elevations to the universal cover $\tilde{Y}$ of components of $\mc{A}$ and $\mc{B}$.
\begin{definition}
  Any elevation of a connected component of $\mc{A}$ to $\tilde{Y}$ is called an \emph{$\mc{A}$--elevation}.  Any geodesic in an $\mc{A}$--elevation is called an \emph{$\mc{A}$--arc}.  We define \emph{$\mc{B}$--elevations} and \emph{$\mc{B}$--arcs} similarly.
\end{definition}

\subsubsection{$\pi_1$--injectivity}
In this subsection we show that under the assumptions \eqref{epsilons} and \eqref{epsilons2},
the space $\Theta$ described in the statement of Theorem
\ref{combotheorem} is $\pi_1$--injective.  Notice that $\Theta$ is itself an NPC space, since it is obtained from NPC spaces by isometrically gluing together convex subspaces (see \cite[Proposition II.11.6]{bridhaef:book}).
It is worth remarking, however, that the inclusion $\Theta \into Y$ is (typically) not a local isometry, and in fact Theorem \ref{combotheorem} is obvious in the case that it is.

We argue by contradiction,
so suppose that $k\neq 1$ is in $K:=\ker(\pi_1\Theta\to \pi_1Y)$.  Let $\tilde{\Theta}$ be the universal cover of $\Theta$.  By possibly
moving basepoints, we can suppose $k$ is represented by a loop $\gamma$ whose
elevation to $\tilde{\Theta}$ is a biinfinite geodesic.  Since $k$ is in $K$, the loop $\gamma$ lifts to the universal cover $\tilde{Y}$ of $Y$.  
The loop is a concatenation of alternating $\mc{A}$--arcs and $\mc{B}$--arcs, for example
\begin{equation} \label{piecewise}
\gamma = a_1b_2\cdots a_{n-1}b_n,
\end{equation}
where each $a_i$ is an $\mc{A}$--arc, and each $b_i$ is a $\mc{B}$--arc.  We regard the expression \eqref{piecewise} as a ``cyclic word'' whose syllables are arcs.  To simplify notation below, the indices $1,\ldots,n$ should be taken to be elements of $\bZ/n\bZ$.

Since $\mc{A}$--elevations and $\mc{B}$--elevations intersect, there may be some choice in the expression \eqref{piecewise}; we always assume the expression is chosen so that the \emph{syllable length} $n$ is minimized.  
Note that $\gamma$ cannot consist of a single $\mc{A}$ or $\mc{B}$--arc, since otherwise some $\mc{A}$ or $\mc{B}$--elevation contains a nontrivial geodesic loop.
Moreover, by moving the basepoint, we can assume that the first arc is an $\mc{A}$--arc, and the last is a $\mc{B}$--arc, as in \eqref{piecewise}.  In other words the syllable length is even.
\begin{lemma}\label{lem:not2}
  The syllable length is not $2$.
\end{lemma}
\begin{proof}
We have already noted that the syllable length is not $1$.  If the syllable length were $2$, we would have a geodesic bigon in $\tilde{Y}$, with one side in some $\mc{A}$--elevation and the other side in some $\mc{B}$--elevation.  This contradicts convexity of the $\mc{A}$ and $\mc{B}$--elevations.
\end{proof}

Since the syllable length is an even number bigger than $2$, there are
at least two $\mc{B}$--arcs (and the same number of $\mc{A}$--arcs) in
$\gamma$.  We describe a way to ``shortcut'' the $\mc{A}$--arcs and
obtain an $\epsilon_B$--local $(1,5\delta)$--quasigeodesic loop from
$\gamma$.  This will contradict Lemma \ref{broken}.

\begin{figure}[htb]
\labellist
\hair 2pt
 \pinlabel {$E_1$} [ ] at 20 228
 \pinlabel {$E_3$} [ ] at 325 256
 \pinlabel {$a_1$} [ ] at 28 173
 \pinlabel {$a_3$} [ ] at 324 228
 \pinlabel {$p_2$} [ ] at 100 225
 \pinlabel {$b_2$} [ ] at 160 235
 \pinlabel {$q_2$} [ ] at 220 250
 \pinlabel {$b_n$} [ ] at 142 50
 \pinlabel {$q_n$} [ ] at 94 85
 \pinlabel {$\alpha_1$} [ ] at 76 161
\endlabellist
\centering
\includegraphics[scale=1.0]{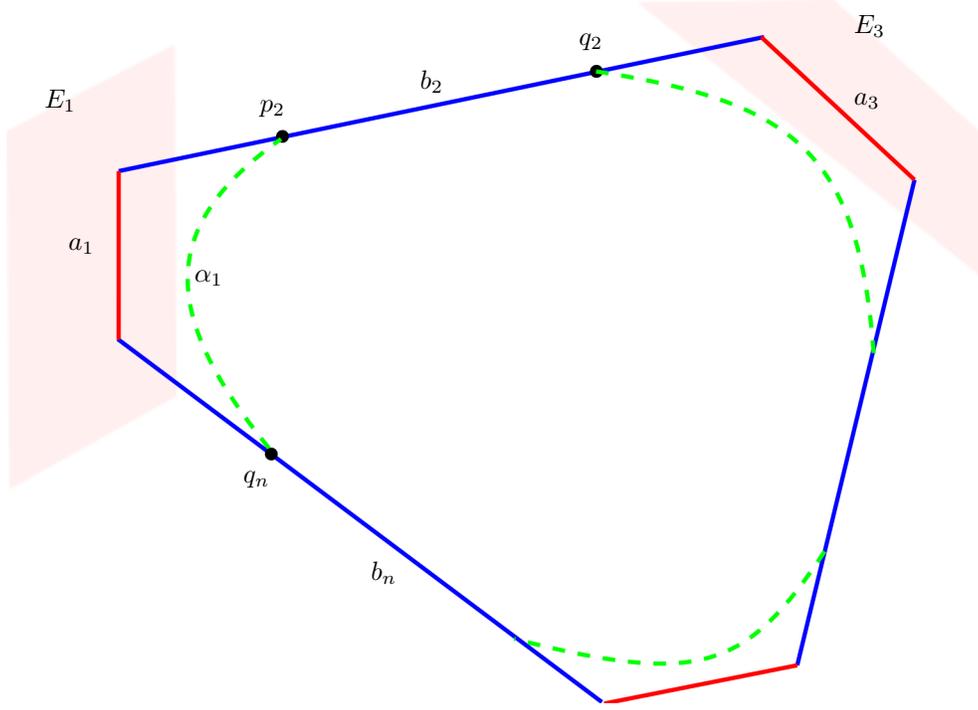}
\caption{Shortcutting around the $\mc{A}$--arcs.}
\label{fig:piecewise}
\end{figure}

Each $\mc{A}$--arc $a_j$ lies in some $\mc{A}$--elevation $E_j$.  Let $a_j$ and $a_{j+2}$ be two consecutive $\mc{A}$--syllables.  The intervening $\mc{B}$--syllable $b_{j+1}$ must leave $E_j$, or otherwise the expression \eqref{piecewise} could be shortened.  Since $E_j$ is convex it follows that $E_{j+2}\neq E_j$.  The $\epsilon_A$--embeddedness of $\mc{A}$ implies that $d(E_j,E_{j+2})\geq 2\epsilon_A$  (A similar argument shows that successive $\mc{B}$--syllables must lie in separate $\mc{B}$--elevations, though there is no useful lower bound on the distance between these elevations.)

Since the elevations $E_j$ and $E_{j+2}$ are distance at least $2\epsilon_A$ from one another, 
 there must be a point $p_{j+1}$ on $b_{j+1}$ satisfying $d(p_{j+1},E_j)=\epsilon_A/2$ and another point $q_{j+1}$ on $b_{j+1}$ satisfying
$d(q_{j+1},E_{j+2})=\epsilon_A/2$.  (See Figure \ref{fig:piecewise}.)
For each odd $j$ we choose a geodesic arc $\alpha_j$ in $\tilde{Y}$ from $q_{j-1}$ to $p_{j+1}$.  For even $j$, we let $\beta_j$ be the subarc of $b_j$ between $p_j$ and $q_j$.  Concatenating all these, we obtain a piecewise geodesic loop
\[ \sigma = \alpha_1\beta_2\cdots\alpha_{n-1}\beta_{n}. \]
Observe $|\beta_j|\geq \epsilon_A$ for each even $j$.

\begin{lemma}\label{alphabig}
 Let $j$ be odd.  Then $|\alpha_j| > 2\epsilon_B$.
\end{lemma}
\begin{proof}
  As argued above, the arc $\alpha_j$ goes between points $q:=
  q_{j-1}$ and $p:= p_{j+1}$ in distinct $\mc{B}$--elevations
  $E_{j-1}$ and $E_{j+1}$.  
  Let $m$ be the midpoint of $\alpha_j$.  By way of contradiction, suppose that $d(m,p)=d(m,q)\leq \epsilon_B$.

\begin{figure}[htbp]
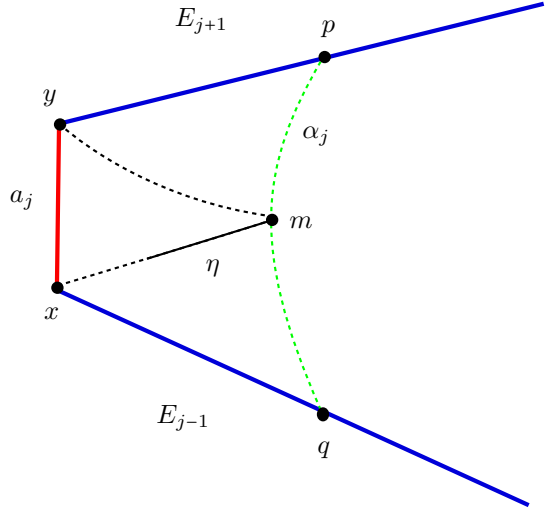
\caption{$\alpha$--arcs are long}
\label{fig:longgreen}
\end{figure}

  Consider a geodesic triangle with base $a_j$ and apex $m$ (see Figure \ref{fig:longgreen}).  Let $x\in E_j\cap E_{j-1}$ and $y\in E_j\cap E_{j+1}$ be the endpoints of $a_j$,
  and let $t = (x,y)_m$ be the Gromov product of the endpoints of $a_j$ with respect to $m$.  
  Let $\eta$ be the initial subsegment of $[m,x]$ of length $t$.  The geodesic segment $\eta$ lies in an $\epsilon_B+\delta$--neighborhood of $E_{j-1}$ and in an $\epsilon_B+2\delta$--neighborhood of $E_{j+1}$.
It follows that the length of $\eta$ is at most $f(\epsilon_B+2\delta)$.  Hence $d(m,E_j)\leq f(\epsilon_B+2\delta)+\delta$, and
\[ \epsilon_A/2 = d(p, E_j) \leq f(\epsilon_B+2\delta)+\delta + \epsilon_B, \]
a contradiction to \eqref{epsilons2}.
\end{proof}

So $\sigma$ is a piecewise geodesic, and each geodesic piece has
length at least $2\epsilon_B$.  We must now bound the Gromov products at the corners.

\begin{lemma}\label{smallcorners} 
  Let $j$ be even.  The Gromov product of $q_{j-2}$ and $q_j$ at $p_j$ is at most $\delta$.  The Gromov product of $p_j$ and $p_{j+2}$ at $q_j$ is at most $\delta$.
\end{lemma}
\begin{proof}
  We prove the first assertion.  The proof of the second assertion is identical up to shuffling labels.  

  We use similar notation to that in Lemma \ref{alphabig}: $p = p_j$, $q = q_{j-2}$, and $q'=q_j$.  Let $x$ be the endpoint of $a_{j-1}$ meeting $b_{j-2}$, and let $y$ be the endpoint meeting $b_j$ (See Figure \ref{fig:smallcorners}).  
\begin{figure}[htbp]
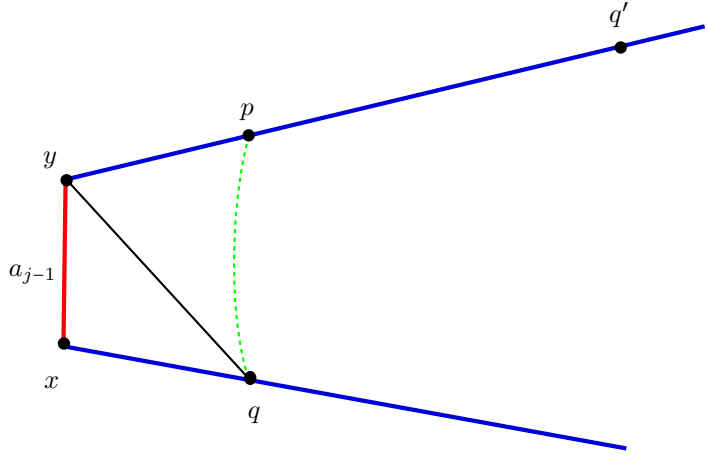
\caption{Showing Gromov products are small.}
\label{fig:smallcorners}
\end{figure}

Because the points $\{y,p,q'\}$ lie on a geodesic, at most one of the quantities $(y,q)_p$, $(q',q)_p$ is greater than $\delta$.  Indeed, consider the two $\delta$--thin triangles with vertex sets $\{p,q,y\}$ and $\{p,q,q'\}$.  Letting $M=\min\{(y,q)_p,(q',q)_p\}$, we see that there are points $p_1\in [p,y]$ and $p_2\in[p,q']$ with $d(p_i,p)=M$ for each $i$, but so that $d(p_1,p_2)\leq2 \delta$.  Since $\{y,p,q'\}$ lie on a geodesic, we have $d(p_1,p_2)=2M$, so $M\leq \delta$.  

We'll show that in fact $(y,q)_p > \delta$, from which we deduce $(q,q')_p\leq \delta$ as required.
Consider the pair of triangles with vertex sets $\{p,q,y\}$ and $\{x,y,q\}$.  There are two cases, depending on the relative sizes of the Gromov products at $y$.
Let $D = (y,q)_p$ be the quantity we are trying to bound from below.
\begin{figure}[htb]
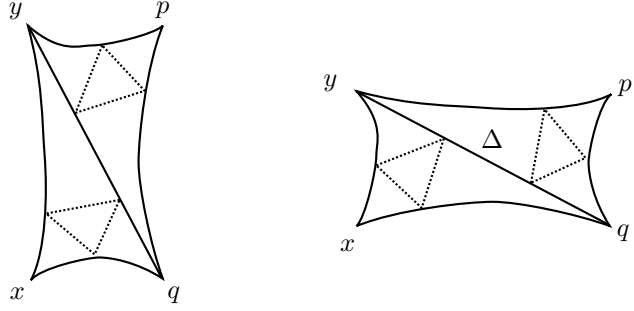
\caption{Two possibilities for the quadrilateral with vertices $(y,p,q,x)$.  Case 1 is on the left, Case 2 on the right.  The sides of the internal dashed triangles have length at most $\delta$.}
\label{fig:twocases}
\end{figure}

\begin{case} 
  $(p,q)_y\leq (x,q)_y$
\end{case}
  In this case (see the left-hand side of Figure \ref{fig:twocases}), there is a path from $p$ to a point on $a_{j-1} \subseteq E_{j-1}$ of length $D+2\delta$.  Since $d(p,E_{j-1}) = \epsilon_A/2$, this implies $D\geq \epsilon_A/2 - 2\delta >\delta$, as required.
\begin{case}
  $(p,q)_y>(x,q)_y$  
\end{case}

  (See the right-hand side of Figure \ref{fig:twocases}.)  Let $\Delta = (p,q)_y-(x,q)_y$ be the difference in the Gromov products at $y$, and note that there is a path from $p$ to a point on $a_j \subseteq E_{j-1}$ of length $D+2\delta+\Delta$.  Moreover, there is a length $\Delta$ subsegment of $[y,q]$ which is distance $\leq \delta$ from both $[y,p]\subseteq E_j$ and $[x,q]\subseteq E_{j-2}$.  But this implies that $\Delta\leq f(\delta)$.
Thus
\[ \epsilon_A/2\leq D+2\delta + \Delta \leq D + 2\delta + f(\delta), \]
  which implies $D \geq \epsilon_A/2 - (2\delta + f(\delta)) > \delta$ as required. 

\end{proof}

We now have that $\sigma$ is a broken geodesic with Gromov products in
the corners of at most $\delta$, and made of pieces of length at
least $2\epsilon_B$.  Since $\sigma$ is a loop, applying Lemma
\ref{broken} with $l=\delta$ we see that $\sigma$ must lie a Hausdorff distance at
most $6\delta$ from any point of $\sigma$.  Since $\sigma$ contains a
geodesic subsegment of length at least $2\epsilon_B > 6\delta$, this
is a contradiction. \qed  

\subsubsection{Quasiconvexity}
The argument for quasiconvexity involves similar ideas.  Let $p$ and $q$ be points in the elevation $\tilde{\Theta}$ to $\tilde{Y}$.  These are connected by a (unique) $\tilde{\Theta}$--geodesic $\gamma$ which can be written as a concatenation of $\mc{A}$--arcs and $\mc{B}$--arcs
\begin{equation}
  \gamma = s_1\cdots s_n
\end{equation}
where $n$ is minimal subject to the constraint that each $s_j$ is a geodesic arc either in a single $\mc{A}$--elevation or a single $\mc{B}$--elevation.  We call $n$ the \emph{syllable length} of $\gamma$.  

Consider the first interior $\mc{B}$ syllable of $\gamma$, which will be $s_2$ or $s_3$, and the last interior $\mc{B}$ syllable which is $s_{n-2}$ or $s_{n-1}$. 
Then we may write $\gamma = \alpha \theta \omega$, where $\alpha$ and $\omega$ consist of at most 2 syllables, and $\theta$ has at most $n-2$ syllables which begin and end with $\mc{B}$ syllables. Let $p'$ and $q'$ be the initial and terminal points of $\theta$. 
As in the proof of $\pi_1$--injectivity, we modify $\theta$ to a path $\sigma$ avoiding the $\mc{A}$--syllables.  This path will not be in $\tilde{\Theta}$, but will be a quasigeodesic which is a controlled distance away from $\tilde{\Theta}$, and we then apply quasigeodesic stability to find that $[p',q']$ is not far from $\sigma$.

Let $[p',q']$ be the geodesic joining $p'$ and $q'$ in $\tilde{Y}$. If we can show that $[p',q']$ is uniformly bounded distance from $\tilde{\Theta}$, then it follows that the geodesic $[p,q]\subset \tilde{Y}$ will be a bounded distance from $\tilde{\Theta}$, since the polygonal arc $\alpha [p',q'] \omega$ is made of at most 5 segments, so that $[p,q]\subset \mathcal{N}_{3\delta}(\alpha [p',q'] \omega)$, and therefore a bounded distance from $\tilde{\Theta}$.  

Let $J = \{j_0,j_0+2,\ldots,j_{\mathrm{max}}-2,j_{\mathrm{max}}\}$ be the set of indices of $\mc{B}$ syllables $s_j$ of $\theta$.  Thus $j_0\in \{2,3\}$, and $j_{\mathrm{max}}\in \{n-2,n-1\}$.  For each $j\in J$, the $\mc{B}$--arc $s_j$ travels between distinct $\mc{A}$--elevations $E_{j-1}$ and $E_{j+1}$, which are distance at least $2\epsilon_A$ away from one another.  There are therefore points $p_j$ and $q_j$ on $s_j$ so that $d(p_j,E_{j-1})=\epsilon_A/2$ and $d(q_j,E_{j-1})= \epsilon_A/2$, and the subsegment $[p_j,q_j]$ has length at least $\epsilon_A$.

We now let $\sigma$ be a broken geodesic with ordered vertices $(p',q_{j_0},\ldots,p_{j_\mathrm{max}},q')$. (See Figure \ref{fig:quasigeodesic}.)

\begin{figure}[htb]
\labellist
\small\hair 2pt
 \pinlabel {$p$} [ ] at 6 7
 \pinlabel {$p'$} [ ] at 27 50
 \pinlabel {$q_3$} [ ] at 72 69
 \pinlabel {$p_5$} [ ] at 156 89
 \pinlabel {$q_5$} [ ] at 190 91
 \pinlabel {$p_7$} [ ] at 275 74
 \pinlabel {$q'$} [ ] at 306 64
 \pinlabel {$q$} [ ] at 346 6
 \pinlabel {$[p,q]$} at 184 10
 \pinlabel {$[p',q']$} at 160  60
\endlabellist
\centering
\includegraphics[scale=1.0]{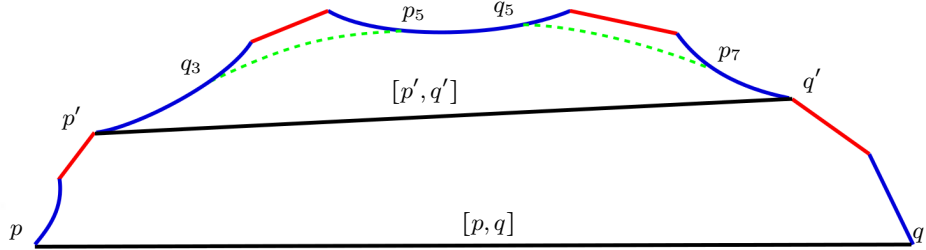}
\caption{Modifying a subpath of $\gamma$ to a quasigeodesic $\sigma$ between $p'$ and $q'$.  Note that the blue segments at the beginning and end of $\gamma$ may be absent, possibly changing the numbering.}
\label{fig:quasigeodesic}
\end{figure}

\begin{lemma}\label{sigmaclosetogamma}
  The path $\sigma$ is contained in a $2\delta$--neighborhood of $\gamma$ (and hence of $\tilde{\Theta}$).
\end{lemma}
\begin{proof}
See Figure \ref{fig:quasigeodesic}.  The path $\sigma$ is built from two kinds of sub-arcs.  The first are paths already in $\gamma$.  The second, pictured as dotted green arcs in the Figure, are geodesics.  These geodesics naturally lie in geodesic triangles or quadrilaterals, the remaining edges of which are in $\gamma$.  Using $\delta$-hyperbolicity, the result is immediate.
\end{proof}

By the same argument used to prove Lemma \ref{smallcorners}, the path $\sigma$ has Gromov products in the corners bounded by $\delta$. 

We may apply Lemma \ref{broken} with $l=\delta$ to see that
$\sigma$ has Hausdorff distance at most $6\delta$ from $[p',q']$. 
Lemma \ref{sigmaclosetogamma} then implies that
$[p',q']$ lies in an $8\delta$--neighborhood of $\tilde{\Theta}$. Then $[p,q]$ lies in an $11\delta$ neighborhood of $\tilde{\Theta}$
as observed above, since it lies in a $3\delta$ neighborhood of $\alpha [p',q']\omega$, which is a polygonal arc made of at most
5 geodesic segments, all of which but $[p',q']$ lie in $\tilde{\Theta}$.

We conclude that
$\tilde{\Theta}$ is $11\delta$--quasiconvex. \qed

\subsection{Proof of Proposition \ref{noaccidents}}
As in the proof of Theorem \ref{combotheorem}, we assume 
\begin{align*}
 \epsilon_B & > 50\delta \mbox{, and }\\
 \epsilon_A & >2 \max\left\{ f(\epsilon_B+2 \delta)+\delta +\epsilon_B \mbox{, }2\epsilon_B \right\}.
\end{align*}

We argue by contradiction, supposing that there is an essential
$\Theta$--loop $\gamma$ which is freely homotopic in $Y$ to an essential loop
$\eta$ in some component $B$ of $\mc{B}$, but that no positive power of $\gamma$ is homotopic in $\Theta$ to a loop in $\mc{B}$.
Adjusting the loops by a homotopy, we may assume that $\eta$ is geodesic and that $\gamma$ is $\Theta$--geodesic, and can be written as a sequence of $\mc{A}$--syllables $a_i$ and $\mc{B}$--syllables $b_i$; for example
\[ \gamma = a_1 b_2\cdots a_{n-1} b_n, \]
where $n$ is the syllable length of $\gamma$, and the expression is
chosen to minimize this length in the $\Theta$--homotopy class of
$\gamma$.
\begin{lemma}
  The syllable length of $\gamma$ is not $1$.
\end{lemma}
\begin{proof}
  If the syllable length of $\gamma$ is $1$, then there are two possibilities, $\gamma\subseteq \mc{A}$ or $\gamma\subseteq\mc{B}$.  The first case is ruled out by the hypothesis that $\mc{A}$ has no accidental $\mc{B}$--loops, the second by assumption that $\gamma$ is not homotopic in $\Theta$ to a loop in $\mc{B}$.
\end{proof}

We suppose therefore that the syllable length of $\gamma$ is greater than $1$, so there is at least one $\mc{B}$--syllable, and at least one $\mc{A}$--syllable.  We can lift the homotopy between $\gamma$ and $\eta$ to the universal cover $\tilde{Y}$, obtaining elevations $\tilde{\gamma}$ and $\tilde{\eta}$ which fellow-travel one another (see Figure \ref{fig:strip}).
\begin{figure}[htb]
\labellist
\small\hair 2pt
 \pinlabel {$\tilde{\eta}$} [ ] at 209 117
 \pinlabel {$\tilde{\gamma}$} [ ] at 300 150
\endlabellist
\centering
\includegraphics[scale=1.0]{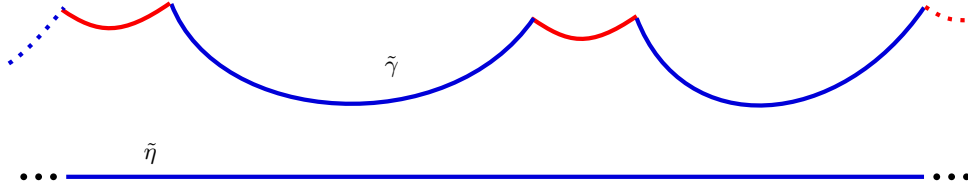}
\caption{Fellow-travelling lifts of $\eta$ and $\gamma$.  $\mc{A}$--arcs are red; $\mc{B}$--arcs are blue.}
\label{fig:strip}
\end{figure}

\begin{lemma}\label{successivedistinct}
  Successive $\mc{B}$--arcs of $\tilde{\gamma}$ are contained in distinct $\mc{B}$--elevations.
\end{lemma}
\begin{proof}
  Suppose that two successive $\mc{B}$--arcs are contained in the same $\mc{B}$--elevation $E$.  Since the elevation is convex, the intervening $\mc{A}$--arc is also contained in $E$.  It follows that $\gamma$ is homotopic in $\Theta$ to a loop with one fewer $\mc{A}$--syllable, contradicting the minimality of the syllable length.
\end{proof}

We can perform the same ``shortcutting'' operation as in the proof of Theorem \ref{combotheorem}, modifying $\tilde{\gamma}$ to a piecewise geodesic $\sigma$ which contains long ($\geq\epsilon_A$) subsegments of the $\mc{B}$--arcs of $\gamma$, and which has Gromov products bounded by $\delta$.
Using Lemma \ref{broken} and $\delta$--hyperbolicity we can see that the broken geodesic $\sigma$ lies within Hausdorff distance $8\delta$ of $\tilde{\eta}$.  
The broken geodesic $\sigma$ contains a subsegment of each $\mc{B}$--arc of $\tilde{\gamma}$ of length at least $\epsilon_A$.  The geodesic $\tilde{\eta}$ is contained in a single $\mc{B}$--elevation.
By Lemma \ref{successivedistinct}, not all the $\mc{B}$--arcs of $\tilde{\gamma}$ are contained in the same $\mc{B}$--elevation as $\tilde{\eta}$.  
We deduce
$ \epsilon_A \leq f(8\delta) $, which contradicts the hypothesis on $\epsilon_A$.
This completes the proof of Proposition \ref{noaccidents}.\qed

\section{Fully $\mc{P}$-elliptic hierarchies} \label{s:vhierarchy}

In this section, we prove Theorem \ref{thm:virtualhierarchy}.  First recall the statement.

\begin{reptheorem}{thm:virtualhierarchy}
  Given $(G,\mc{P})$ relatively hyperbolic, with $G$ hyperbolic and virtually special, there is some $(G',\mc{P}')\dotnorm (G,\mc{P})$ which has a malnormal quasiconvex fully $\mc{P'}$-elliptic hierarchy terminating in $\mc{P}'$.
\end{reptheorem}

We describe a fairly general procedure for generating
relative hierarchies \emph{of spaces} from cube complex pairs.  Then
we talk about how to ensure such a hierarchy is faithful,
quasiconvex, and malnormal, by passing to carefully chosen finite covers.

\subsection{Augmented cube complexes}
We first establish some terminology useful for dealing with cube complex pairs.

Let $X$ be a NPC cube complex, and $\mc{Z} = \bigsqcup_{i=1}^n Z_i$ a union of connected NPC cube complexes which admits a locally isometric immersion $\Phi = \bigsqcup_{i=1}^n\phi_i\co \mc{Z}\to X$.  We'll call $(X,\mc{Z})$ a \emph{cube complex pair}.
Then the mapping cylinder of $\Phi$,
\begin{equation} 
C_\Phi = \sfrac{\left.X\sqcup \left(\mc{Z}\times [0,1]\right)\right.}{\{ (z,1)\sim \Phi(z) \}}
\end{equation}
naturally has the structure of an NPC cube complex.  
Call such a complex the \emph{augmented cube complex} based on the pair $(X,\mc{Z})$.

There are canonical inclusions $X\hookrightarrow C_\Phi$ and
$\mc{Z}\hookrightarrow C_\Phi$ (this second via $z\mapsto (z,0)$).
The subset $X\subseteq C_\Phi$ is a deformation retract of $C_\Phi$.
The components $Z_i$ of $\mc{Z}\subseteq C_\Phi$ are referred to
as \emph{peripheral subcomplexes}.  The hyperplanes $Z_i\times
\{\frac{1}{2}\}$ are called \emph{peripheral hyperplanes}, and all
other hyperplanes of $C_\Phi$ are called 
\emph{non-peripheral hyperplanes}.

Any cover $\tilde{X}\to X$ also gives rise to a cover $C_{\tilde{\Phi}}\to
C_\Phi$, where $\tilde\Phi\co \tilde{\mc{Z}}\to \tilde{X}$ is made up of
all the elevations of the maps $\phi_i$.  

\begin{definition}[Augmented Hyperplanes]
  Let $\Phi\co \mc{Z}\to X$ be a locally isometric immersion of NPC
  cube complexes, and let $C = C_{\Phi}$ be the augmented complex.
  Let $W$ be a non-peripheral hyperplane of $C$.  The \emph{augmented
    hyperplane} $A(W)$ is the component of $W\cup \mc{Z}$ containing
  $W$.
\end{definition}
In other words the augmented hyperplane $A(W)$ is the union of $W$
with any components of $\mc{Z}$ which meet $W$.  In general, augmented
hyperplanes are not particularly well behaved.  They are not convex
subsets of the augmented complex, are not locally separating, and may
not be $\pi_1$--injective.  We will nonetheless be able to find
situations in which they give a faithful hierarchy.

\subsection{Double-dot hierarchy}\label{doubledothierarchy}
In this subsection we start with a NPC cube complex $X$, together with
some collection of subcomplexes $\mc{Z}$, and construct a particular
hierarchy on a space homotopic to a ``generalized double cover'' of
$X$.  The stages of the hierarchy (though not its terminal spaces) depend on an ordering of the hyperplanes of $X$.
In the rest of the paper, the hierarchy is referred to as the
\emph{double dot hierarchy for the pair $(X,\mc{Z})$}.

\begin{definition}\label{ddotcover} (cf. \cite[Construction 9.1]{Wise}) 
  Let $X$ be any cube complex.  Any hyperplane $W\subset X$ gives rise
  to a map from $i_W\co \pi_1X\to \bZ/2$, measuring the mod-$2$
  intersection between a loop and the hyperplane.  Let $\mc{W}$ be the set of embedded, $2$--sided, nonseparating hyperplanes in $X$.
  The
  \emph{double-dot cover} $\ddot{X}\to X$ is the cover corresponding
  to the kernel of the map 
  \[h_\mc{W} = \oplus i_W\co \pi_1X\to \bigoplus_{W\in\mc{W}}\bZ/2.\]
\end{definition}

Given any augmented complex $C_\Phi$, we now describe a ``relative''
hierarchy of spaces structure on the double-dot cover using
the augmented hyperplanes of $C_\Phi$.  An example of an augmented cube complex is shown in Figure \ref{fig:doubledot1}, together with its double-dot cover.  Parts of the double-dot hierarchy for this complex are shown in Figures \ref{fig:hierarchy1} and \ref{fig:hierarchy2}.

\begin{figure}[htb]
\centering
\includegraphics[width=\textwidth]{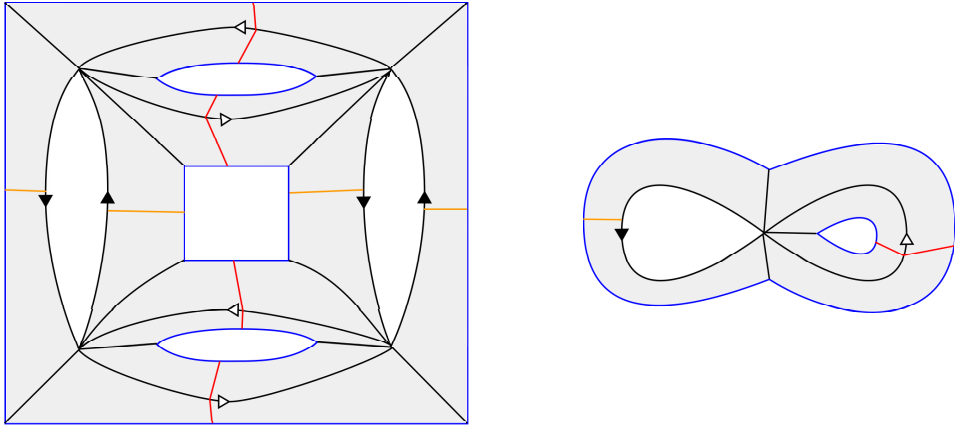}
\caption{An example of an augmented complex (at right) and its double-dot cover.  $X$ is a figure-eight, and $\mc{Z}$ consists of two circles.  There are two nonseparating hyperplanes, red and orange.}
\label{fig:doubledot1}
\end{figure}

Let $C = C_\Phi$ be an augmented cube complex, where $\Phi\co
\mc{Z}\to X$ is a locally isometric immersion of cube complexes.  Let
$\mc{W}$ be the set of embedded, $2$--sided, nonseparating hyperplanes
of $C$.  
Choose an order $(W_1,\ldots,W_n)$ for these
hyperplanes.  
For $i\in \{1,\ldots,n\}$, let $A_i$ be the augmented hyperplane $A(W_i)$.

Let $c\co \ddot{C}\to C$ be the double-dot cover, and let
$\ddot{\mc{Z}}= c^{-1}(\mc{Z})$.  Fix a basepoint $\tilde{p}\in
\ddot{C}$ so that $p = c(\tilde{p})$ lies in the complement of
$\bigcup\mc{W}$.  
The complementary components of
$c^{-1}\left(\bigcup\mc{W}\right)$ are labeled by elements of
$\bigoplus_{\mc{W}}(\bZ/2)$.  
For  $\mathbf{t}\in
\bigoplus_{\mc{W}}(\bZ/2)$, let $V_\mathbf{t}$ be the closure of the
union of components labeled by $\mathbf{t}$.  Note that $V_\mathbf{t}$ need not be connected.
The terminal vertex spaces of our
hierarchy are components of $V_\mathbf{t}\cup\ddot{\mc{Z}}$, where $\mathbf{t}$ ranges over $\bigoplus_{\mc{W}}(\bZ/2)$.

The order on $\mc{W}$ determines a hierarchy of spaces $\mc{H}$ as follows:
For each $i$, let $\mc{W}_i = \{W_1,\ldots,W_i\}$, and let $M_i = \bigoplus_{\mc{W}_i}\bZ/2$.  The complementary components of $\cup\mc{W}_i$ are labeled by elements of $M_i$.
For each $\mathbf{t}\in M_i$, let $K_\mathbf{t}$ be the closure of the part labeled $\mathbf{t}$.
We define the \emph{$\mathbf{t}$--vertex spaces} to be those components of $K_\mathbf{t}\cup \ddot{\mc{Z}}$ which intersect $K_\mathbf{t}$.  The vertex spaces at level $i$ are the $\mathbf{t}$--vertex spaces, for $\mathbf{t}$ ranging over $M_i$.  

The edge spaces at level $i$ are (some of the) components of pairwise intersections of vertex spaces.  We now describe these intersections.

\begin{figure}[htb]
\centering
\includegraphics[width=\textwidth]{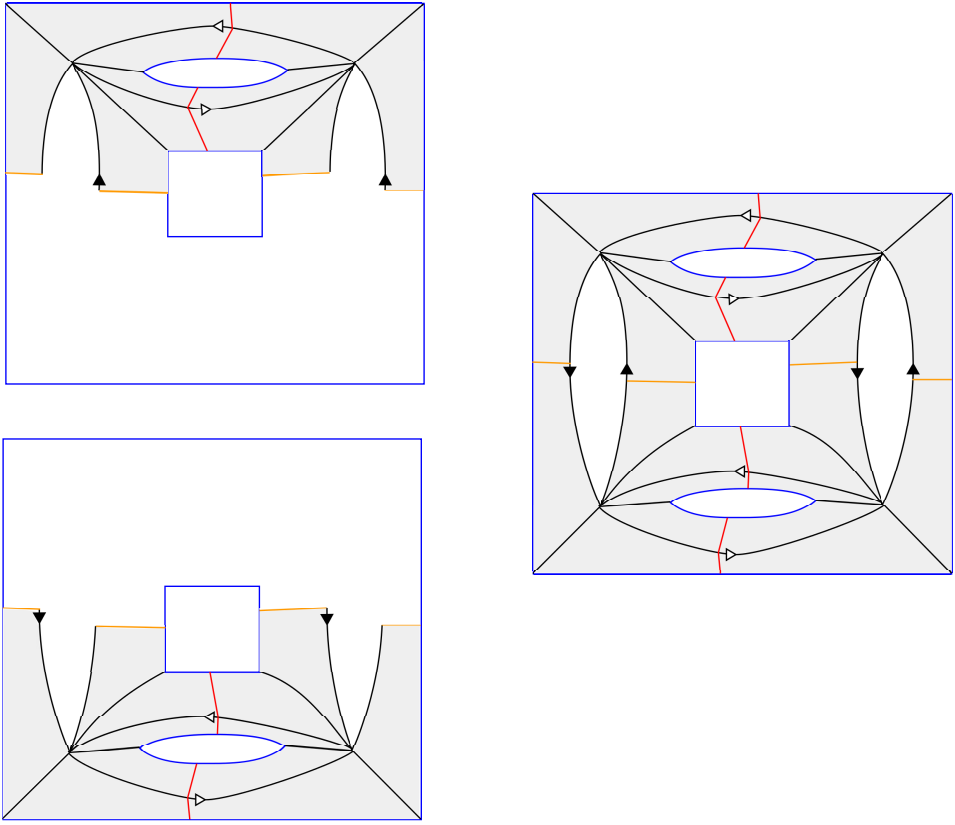}
\caption{Level $1$ of a double-dot hierarchy, cutting along the elevated orange augmented hyperplane.  There are two edge spaces, each consisting of two elevations of the orange hyperplane together with an elevation of a blue peripheral circle.}
\label{fig:hierarchy1}
\end{figure}
\begin{figure}[htb]
\centering
\includegraphics[width=\textwidth]{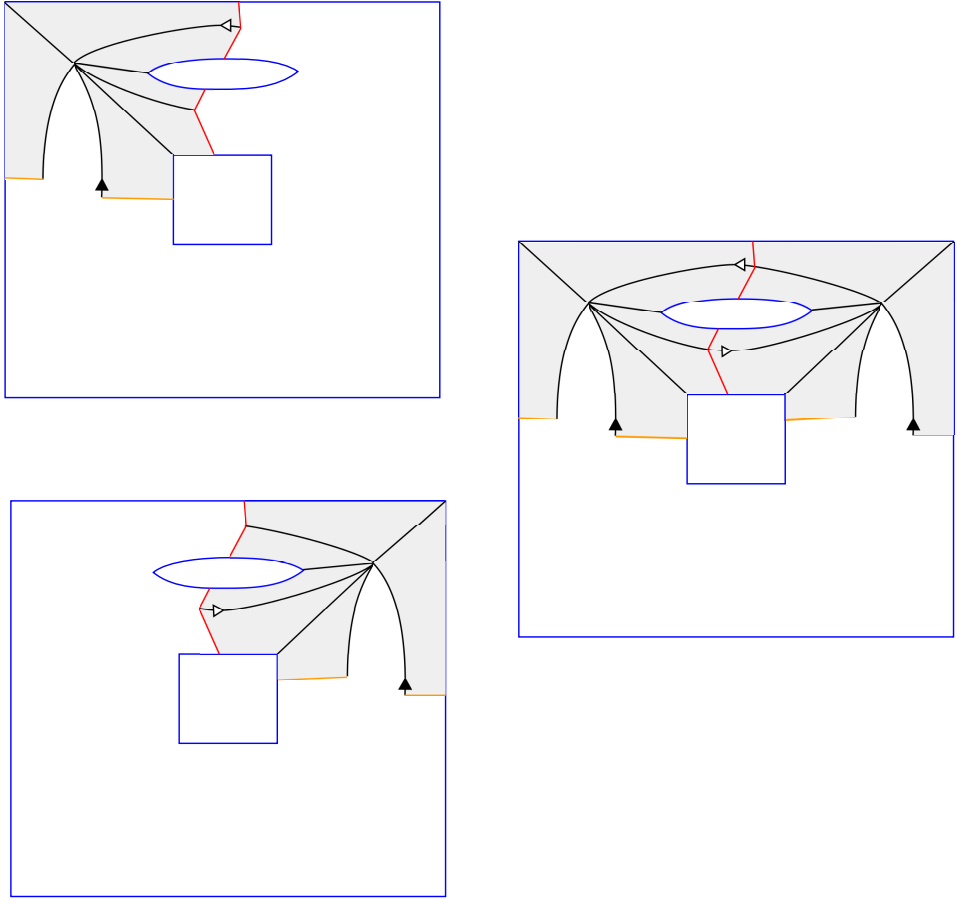}
\caption{Level $2$ of a double-dot hierarchy, cutting one of the vertex spaces along the elevated red augmented hyperplane.  There is a single edge space.}
\label{fig:hierarchy2}
\end{figure}

\begin{definition}
  Let $A$ be the closure of a component of $c^{-1}(W_i)\setminus \bigcup_{j<i} c^{-1}(W_j)$ in $\ddot{C}$.  We call $A$ a \emph{partly-cut-up elevation of $W_i$}.
\end{definition}

The following lemma is immediate from the construction.
\begin{lemma}
  Any two vertex spaces at level $i$ are either disjoint or intersect
  in a union of peripheral complexes and disjoint partly-cut-up
  elevations of $W_i$.
\end{lemma}

We now describe the graph of spaces $(\Gamma,\mc{X})$ at level $i$, associated to a
vertex space $V$ at level $i-1$.  The space $V$ is a $\mb{t}$--vertex
space, for some $\mb{t}\in M_{i-1}$.  Under the canonical projection
$M_i\to M_{i-1}$, the fiber above $\mb{t}$ consists of two elements
$\{\mb{t}^+,\mb{t}^-\}$, so $V$ is a union of $\mb{t}^+$--vertex
spaces $\mc{V}^+=\{V_1^+,\ldots,V_p^+\}$ and $\mb{t}^-$--vertex spaces $\mc{V}^- = \{V_1^-,\ldots V_m^-\}$.  These form the vertex spaces of
the graph of spaces associated to $V$.  No two $\mb{t}^+$--vertex
spaces intersect nontrivially, and neither do any $\mb{t}^-$--vertex
spaces.  Thus the incidence graph of the set of vertex spaces is bipartite.  We modify the incidence graph slightly to a graph $\Gamma$ by repeating
edges representing multiple components of intersection.  Thus the unoriented edges of $\Gamma$ are in bijective correspondence with the components of 
\[ \left(\bigcup \mc{V}^+\right)\cap \left(\bigcup\mc{V}^-\right).\]
The edge-space-to-vertex-space maps are the inclusion maps, and the homotopy equivalence
$ R(\Gamma,\mc{X}) \to V $
is given in the obvious way, by inclusion on the vertex spaces, and
projection of the mapping cylinders to their images.

\subsection{Terminal spaces}
In general, the hierarchy described in Section \ref{doubledothierarchy} above may
not be faithful, and even if it is, may not be quasiconvex, malnormal,
or have reasonable terminal spaces.  However, if every hyperplane of $X$ is embedded, nonseparating and two-sided, the terminal spaces are particularly nice:
\begin{lemma}\label{lem:terminalspaces}
Suppose that the elements of $\mc{Z}$ are embedded and locally isometrically embedded in $X$.  Let $C$ be the augmented cube complex coming from the pair
  $(X,\mc{Z})$, and let $\ddot{C}$ be the double-dot cover.  Let
  $\ddot{\mc{Z}}$ be the union of the peripheral subcomplexes of
  $\ddot{C}$.  Suppose that every hyperplane of $X$ is embedded, $2$--sided, and nonseparating.  

  Let $Y$ be a terminal space of the double-dot hierarchy.  Then $Y$ has a graph of spaces structure $(\Gamma,\mc{Y})$, where 
  \begin{enumerate}
  \item $\Gamma$ is bipartite, with red and black vertices;
  \item if $v$ is a black vertex, $\mc{Y}(v)$ is contractible;
  \item if $v$ is a red vertex, $\mc{Y}(v)$ is a component of $\ddot{\mc{Z}}$; and
  \item every edge space is contractible.
  \end{enumerate}
\end{lemma}
\begin{proof}
The hypotheses imply that the set of hyperplanes $\mc{W}$ from Definition \ref{ddotcover} is the entire collection of non-peripheral hyperplanes of $C$.  

After cutting along each of the $W_i$ in turn (for varying $i$), what remains are:  parts of $X$ cut along all non-peripheral hyperplanes, joined (via parts of the cylinders in the trivial mapping cylinder of the augmentation) to components of $\ddot{\mc{Z}}$.  The graph of spaces decomposition of $Y$ is obtained by cutting along the (remaining parts of the) peripheral hyperplanes.  Since we have now cut along \emph{all} the hyperplanes of $\ddot{C}$, the remaining parts which do not intersect the peripheral sub-complexes are cubical polyhedra (in the language of \cite{VH}) and are contractible.  The other parts are homotopy equivalent to components of $\ddot{\mc{Z}}$.  Finally, the edge spaces are cubical polyhedra in the peripheral hyperplanes, and hence are also contractible.
\end{proof}

\begin{corollary}\label{cor:terminalgroups}
  With the same assumptions as Lemma \ref{lem:terminalspaces}, the fundamental group of a terminal space is a free product $(\ast_{i=1}^p G_i)\ast F$ where $F$ is finitely generated free, and each $G_i$ is the fundamental group of some component of $\ddot{\mc{Z}}$.
\end{corollary}

\subsection{Finding a faithful hierarchy}
In order to turn the topological statements of the last subsection
into group theoretic statements, we must find conditions under which
the double-dot hierarchy is faithful.  This uses the combination
theorem from Section \ref{sec:combination}.

\subsubsection{Lemmas about $R$--embedded subsets}
The following lemma is straightforward.
\begin{lemma}\label{lem:coverstillembedded}
  Suppose that $A$ is $R$--embedded in $B$.  Let $\pi\co \tilde{B}\to B$ be a finite cover, and let $\tilde{A} = \pi^{-1}(A)$.  Then $\tilde{A}$ is $R$--embedded in $\tilde{B}$.
\end{lemma}
We'll need to know that certain intersections of $R$--embedded sets are $R$--embedded (Lemma \ref{lem:embeddedintersection}).  First we need to understand the fundamental groups of those intersections.
\begin{lemma}\label{lem:pi1intersection}
  Let $A$ and $B$ be locally convex subsets of the NPC space $C$.  Let $I$ be a component of $A\cap B$, and let $p\in I$.  Then $\pi_1(I,p)=\pi_1(A,p)\cap \pi_1(B,p)$.
\end{lemma}
\begin{proof}
  If $\sigma\in \left(\pi_1(A,p)\cap \pi_1(B,p)\right)\setminus \pi_1(I,p)$, then represent $\sigma$ as a based geodesic loop $\gamma_A$ in $A$ and as another such loop $\gamma_B$ in $B$.  Unless they coincide, the homotopy between them lifts to a nondegenerate geodesic bigon in the CAT$(0)$ universal cover of $C$.  Therefore $\gamma_A = \gamma_B \subseteq A\cap B$.
\end{proof}
\begin{lemma}\label{lem:embeddedintersection}
  Let $R>0$, and suppose $A$ and $B$ are $R$--embedded locally convex subsets of the NPC space $C$.  Then $A\cap B$ is $R$--embedded.
\end{lemma}
\begin{proof}
  We first show that components of $A\cap B$ are $R$--embedded, and then show they are further than $2R$ from one another.

  Let $I$ be a component of $A\cap B$.  Then $I$ is contained in $A_0\cap B_0$, where $A_0$ is a connected component of $A$ and $B_0$ is a connected component of $B$.
  Choosing $p\in I$ determines a diagram of covers
\cd{ & C^{I}\ar[dl]_{\phi_{A_0}}\ar[dd]^{\pi_{I}}\ar[dr]^{\phi_{B_0}} & \\
     C^{A_0}\ar[dr]_{\pi_{A_0}} & & C^{B_0}\ar[dl]^{\pi_{B_0}} \\
     & C &}
with canonical inclusions $N_R({A_0})\subseteq C^{A_0}$, $N_R({B_0})\subseteq C^{B_0}$, and $N_R(I)\subseteq C^I$.  We suppose by contradiction that there are two points $x\neq y$ in $N_R(I)$ so that $\pi_I(x)=\pi_I(y)$.  Since $\phi_{A_0}(N_R(I))\subseteq N_R({A_0})$ and $A$ is $R$--embedded, it must be the case that $\phi_{A_0}(x) = \phi_{A_0}(y)$.  Similarly $\phi_{B_0}(x)=\phi_{B_0}(y)$.  Let $\sigma_x$ be an arc joining $x$ to the lifted basepoint $\tilde{p}$ in $I$, and let $\sigma_y$ join $y$ to $\tilde{p}$.  Then $\gamma = \pi_I(\bar{\sigma_x}\sigma_y)$ is a loop in $C$ based at $p$ which lifts to both $C^{A_0}$ and $C^{B_0}$, but not to $C^I$, contradicting Lemma \ref{lem:pi1intersection}.

  Now suppose that $I_1$ and $I_2$ are two different components of $A\cap B$.  
  By way of contradiction, suppose $\sigma$ is an arc joining $I_1$ to $I_2$ of length strictly less than $2R$.  
  For $i=1,2$, 
  let $A_i$ be the connected component of $A$ containing $I_i$, and let $B_i$ be the connected component of $B$ containing $I_i$.  Since $A$ is $R$--embedded and $\sigma$ has length less than $2R$, we must have $A_2=A_1$; similarly $B_2=B_1$.  Again using $R$--embeddedness, $\sigma\cup A_1$ lifts to $C^{A_1}$.  It follows that $\sigma$ is homotopic rel endpoints to an arc $\tau_A$ in $A$.  Since $A$ is locally convex, we may choose $\tau_A$ to be geodesic.  Similarly, $\sigma$ is homotopic rel endpoints to a geodesic $\tau_B$ in $B$.  But this gives rise to a nondegenerate geodesic bigon in the universal cover $\tilde{C}$, contradicting the CAT$(0)$ inequality.
\end{proof}

\subsubsection{Criterion for the double-dot hierarchy to be faithful}

We are particularly interested in subspaces coming from cutting a cube complex up along walls.
Let $Y$ be an NPC cube complex, and let $\ddot{Y}\to Y$ be the double-dot cover.  Let $W$ be a nonseparating $2$--sided embedded hyperplane of $Y$, and let $\ddot{W}$ be the preimage of $W$ under the double-dot cover.  Then $\ddot{W}$ cuts $Y$ into two (not necessarily connected) parts, $V^+$ and $V^-$.
\begin{lemma}\label{lem:embeddedhalves}
  If $W$ is $R$--embedded then $V^+$ is $R$--embedded, as is $V^-$.
\end{lemma}
\begin{proof}
  Each component of $V^+$ is bounded by a collection of elevations of $W$.  Lemma \ref{lem:coverstillembedded} implies that the union of these elevations is $R$--embedded, so $V^+$ is as well.
\end{proof}

Combining Lemmas \ref{lem:embeddedhalves} and \ref{lem:embeddedintersection} we obtain:
\begin{corollary}\label{lem:embeddedchunks}
  Let $\mc{W}$ be a collection of $2$--sided, embedded, nonseparating, $R$--embedded hyperplanes in a cube complex $Y$, and let $\mb{t}\in \oplus_{\mc{W}}\bZ/2$.  If $V_{\mb{t}}$ is that part of the double-dot cover labeled by $\mb{t}$ then $V_{\mb{t}}$ is $R$--embeddded.
\end{corollary}

\begin{definition}
  Let $\phi\co Z\to X$ be a locally isometric immersion.  We say that
  $\phi$ (or sometimes $Z$) is \emph{superconvex} if whenever
  $\tilde{Z}$ is an elevation of $Z$ to the universal cover of $X$,
  then $\tilde{Z}$ contains every biinfinite geodesic which lies in a bounded
  neighborhood of $\tilde{Z}$. 
  \end{definition}
\begin{lemma}\label{lem:haglund}\cite{Haglund08} (cf \cite{SageevWise})
  Let $G$ be hyperbolic and the fundamental group of an NPC cube complex $X$.  Let $H<G$ be a quasiconvex subgroup.  Then there is a superconvex immersion $\phi\co Z\to X$ with $\phi_*(\pi_1(Z)) = H$.
\end{lemma}
\begin{proof}
Let $\Lambda_H$ be the limit set of $H$ in $\partial G = \partial X$.  By quasi-convexity and $\delta$-hyperbolicity, any geodesic joining elements of $\partial H$ stays uniformly close to $H$.  Thus, there is a compact set $K$ in $\tilde{X}$ so that $H.K$ contains all such geodesics. It follows by $\delta$-hyperbolicity of $\tilde{X}$ that quasi-convexity and combinatorial quasi-convexity (in the sense of \cite[Definition 2.24]{Haglund08}) are the same notion.  Therefore, by \cite[Theorem H]{Haglund08} the combinatorial convex hull $Y$ of $H.K$ is $H$-cocompact.  We take $Z = Y/H$.
\end{proof}

\begin{definition}
  Let $(X,\mc{Z})$ be a cube complex pair, so that the immersion
  $\Phi\co \mc{Z}\to X$ is superconvex on each component of $\mc{Z}$.
  Then $(X,\mc{Z})$ is a \emph{superconvex pair}.
\end{definition}

In order to ensure our edge spaces don't have any accidental parabolics, we need the following lemma:

\begin{lemma}\label{lem:noaccidents}
  Let $(X,\mc{Z})$ be a superconvex pair, with each component of $\mc{Z}$ embedded, and let $C$ be the corresponding augmented cube complex.  For $n\geq 1$ let $\{W_1,\ldots,W_n\}$ be a collection of embedded, $2$--sided, nonseparating hyperplanes of $C$.  Let $P$ be a component of $W_n\setminus\bigcup_{i<n}W_i$.  Then $P$ has no accidental $\mc{Z}$--loops.
\end{lemma}
\begin{proof}
  Recall that the augmented space $C$ is the mapping cylinder of a
  locally isometric immersion $\Phi$ from $\mc{Z} = \bigsqcup_{j=1}^k
  Z_j$ to $X$, where each $Z_j$ is a connected NPC cube complex.  Let
  $\phi_j = \Phi|Z_j$.  Consider a cube $\tau$ of $C$ which
  intersects, but isn't contained in $X$.  This cube can be identified
  with $\sigma\times[0,1]$ where $\sigma=\sigma\times\{0\}$ is a cube
  of $Z$, and $\sigma\times\{1\} = \Phi(\sigma)$.  We observe that
  under this identification, $P\cap \tau = (P\cap \sigma)\times[0,1]$.
  Since $Z_j$ is embedded, we get a copy of 
  $Z_j\times[0,1]\subseteq C$, and the foregoing argument shows that
  the intersection of $P$ with this product is exactly $(P\cap Z_j\times\{0\})\times[0,1]$.

  Now, let $\alpha\subseteq P$ be a geodesic loop which is homotopic
  (in $C$) to a geodesic loop $\alpha'$ in $Z_j\subseteq \mc{Z}$.  
  
  Suppose first that $\alpha$ is contained in $C\setminus X$.  It is
  then possible to homotope $\alpha$ along the product structure into
  $\mc{Z}$.  Since the intersection of $P$ with the product structure
  is itself a product, this homotopy takes place inside $P$, so
  $\alpha$ is not an accidental $\mc{Z}$--loop.
  
  We may therefore assume that $\alpha$ intersects $X$.  We claim that $\alpha$ must be contained in $X$.  Suppose not.  By reparameterizing the geodesic loop $\alpha\co [0,L]\to C$, we may suppose that $\alpha(0)=\alpha(L)\in X$, but that $\alpha(t)\notin X$ for small positive $t$.  Let $R$ be the smallest positive number so that $\alpha(R)\in X$, and let $\beta= \alpha|_{[0,R]}$.  Then $\beta$ has image in some $Z_j\times[0,1]$, so we may write $\beta(t) = (z(t),s(t))$ for $z(t)\in Z_j$, $s(t)\in [0,1]$.   The projection $s(t)$ must also be locally geodesic, so the quantity $s'(t)$ is constant.  Since $s(0)=s(R) = 1$, we have $s'(t)=0$ for all $t\in [0,R]$, contradicting the supposition that $\alpha(t)\notin X$ for small positive $t$.

  Lifting to the universal cover, we obtain a biinfinite geodesic
  $\tilde{\alpha}\subseteq \tilde{X}$ which lies in a bounded
  neighborhood of the image of some elevation of $\phi_j\co Z_j\to X$.  Since $\phi_j$ is
  superconvex, $\tilde{\alpha}$ actually lies in this elevation.  But
  this means that $\alpha$ is in $P\cap \phi_j(Z_j)= P\cap
  (Z_j\times\{1\})$.  Since $P$ intersects $Z_j\times[0,1]$ in a
  product, $\alpha$ is homotopic in $P$ into $Z_j$, and therefore
  into $\mc{Z}$.  Thus $\alpha$ is not an accidental
  $\mc{Z}$--loop.
\end{proof}

\begin{theorem}\label{goodhierarchy}
  Suppose that $(X,\mc{Z})$ is an NPC cube complex pair, and suppose
  that $C$ is the associated augmented complex.  Suppose that the
  universal cover of $C$ is $\delta$--hyperbolic, and that there exists some function $f$ so that the family $\mc{Z}$ is $f$--separated.  Let $\epsilon_A$ and $\epsilon_B$ be constants sufficient for the Combination Theorem \ref{combotheorem}.
\begin{enumerate}
  \item\label{good:qc}  Suppose that each nonseparating embedded $2$--sided hyperplane of
  $C$ is $\epsilon_A$--embedded, and that each component $Z$ of
  $\mc{Z}$ is $\epsilon_B$--embedded.  Then the double-dot hierarchy
  of $\ddot{C}$ is faithful, and the corresponding hierarchy of fundamental
  groups is quasiconvex.

  \item\label{good:noaccidents} Suppose further that the pair $(X,\mc{Z})$ is superconvex.  Then the edge spaces of the double-dot hierarchy have no accidental $\ddot{\mc{Z}}$--loops.
\end{enumerate}
\end{theorem}
\begin{proof}
  For the first conclusion, it suffices to show all the edge spaces are $\pi_1$--injective, with
  quasi-convex elevations to the universal cover $\tilde{C}$ of $C$.

  Let $\{W_1,\ldots,W_n\}$ be the ordered list of embedded,
  nonseparating, $2$--sided hyperplanes in $C$.  Let $V$ be a vertex
  space of the hierarchy at level $i-1$.  Then the non-peripheral part
  of $V$ is marked by some $\mb{t}\in M_{i-1} =
  \bigoplus_{\{W_1,\ldots,W_{i-1}\}}\bZ/2$.  Let $V_{\mb{t}}$ be the
  part of $\ddot{C}$ marked by $\mb{t}$.  By Corollary \ref{lem:embeddedchunks}, the set $V_{\mb{t}}$ is $\epsilon_A$--embedded.

  The vertex space $V$ is a
  component of $V_{\mb{t}}\cup \ddot{\mc{Z}}$.  Let
  $\{\mb{t}^+,\mb{t}^-\}$ be the fiber above $\mb{t}$ under the
  projection $M_i\to M_{i-1}$.

  Let $E$ be an edge space at level $i$ of the hierarchy.  Then $E$ is
  a component of the intersection of two vertex spaces $V^+$ and
  $V^-$, marked by $\mb{t}^+$ and $\mb{t}^-$, respectively.  Let $\ddot{W_i}$ be the preimage of $W_i$ in the double-dot cover $\ddot{C}$.  By Lemma \ref{lem:coverstillembedded}, $\ddot{W_i}$ is $\epsilon_A$--embedded.  

  Let
  $\mc{A} = V_{\mb{t}}\cap \ddot{W_i}$.  By Lemma \ref{lem:embeddedintersection}, $\mc{A}$ is $\epsilon_A$--embedded.

  Let $\mc{B} = \ddot{\mc{Z}}\subseteq \ddot{C}$.  Then $E$
  is a component of $\mc{A}\cup \mc{B}$.  
  The first conclusion
  of the theorem follows from the
  Combination Theorem \ref{combotheorem} for $\mc{A}$, $\mc{B}$, and setting $\Theta=E$.

  For the second conclusion we argue as follows.  At level $i$ of the hierarchy the edge spaces are components of $\mc{A}\cup\mc{B}$ where the components of $\mc{A}$ are elevations of components of $W_i\setminus \bigcup_{j<i}W_j$. 
  Lemmas \ref{lem:noaccidents} and \ref{lem:liftedaccidents} imply that the components of $\mc{A}$ have no accidental $\mc{B}$--loops.  We can therefore apply Proposition \ref{noaccidents} to deduce that $E$ has no accidental $\ddot{\mc{Z}}$--loops.
\end{proof}

\subsection{Scott's Criterion and separability}

We briefly recall Scott's Criterion \cite{Scott78} about separability as it is used multiple times in the proof of Theorem \ref{thm:virtualhierarchy} as an application of separability.  We ignore basepoints in what follows since doing this does not cause any ambiguity or confusion.

Suppose that $X$ is a connected complex and that $H$ is a subgroup of $\pi_1(X)$.  Let $X^H$ be the cover of $X$ corresponding to $H$.  Scott's Criterion states that $H$ is separable in $\pi_1(X)$ if and only if 
the following condition holds:

For every finite sub-complex $\Delta$ of $X^H$ there exists an intermediate covering
\[	X^H \to X_\Delta \to	X	\]
so that $X_\Delta \to X$ is a finite-degree covering and $\Delta$ embeds in $X_\Delta$.

We use this (often implicitly) to promote certain immersions to embeddings in finite covers, and for other purposes that are made clear by this criterion.

\subsection{Proof of Theorem \ref{thm:virtualhierarchy}} \label{sec:vhier}
In the following we make repeated use of Scott's criterion, together with 
the fact (referred to as \emph{QCERF})
that quasiconvex subgroups of virtually special groups are separable (Theorem \ref{qcthensep}).

\medskip\noindent\textbf{Step 1:}
We start with a relatively hyperbolic pair $(G,\mc{P})$ so that $G$ is
hyperbolic and virtually special.  Since $G$ is virtually special,
there is some $(G_0,\mc{P}_0)\dotnorm (G,\mc{P})$ (see Notation \ref{subpair})
so that $G_0 = \pi_1 X_0$, where $X_0$ is
a special cube complex.  
\begin{lemma}\label{nontrivialsplittings}
  If $X_0$ is a compact special cube complex, then some isometrically embedded subcomplex $Y$ of $X_0$ is homotopy equivalent to $X_0$ and has the property that every hyperplane of $Y$ gives a nontrivial splitting of $G_0 = \pi_1 X_0$.
\end{lemma}
\begin{proof}
  Since $X_0$ is special, every hyperplane is embedded and two-sided.  Thus each hyperplane gives some one-edge splitting of $G_0$.  Suppose $W$ is a hyperplane giving a trivial splitting of $G_0$, with open cubical neighborhood $N(W)$.  Since the group-theoretic splitting is trivial, $X_0\setminus N(W)$ must have two components, at least one of which is homotopy equivalent to $X_0$.  Let $X'$ be a component of $X_0\setminus N(W)$ homotopy equivalent to $X_0$.  The total number of cubes (of all dimensions) in $X'$ is less than the total number of cubes in $X_0$.  Thus we can only perform this procedure finitely many times, eventually arriving at a subcomplex which every hyperplane splits nontrivially.
\end{proof}

We may therefore suppose (possibly replacing $X_0$ by a
subcomplex)
every hyperplane of
$X_0$ gives a nontrivial splitting of $G_0$.  Any finite cover of
$X_0$ therefore has the same property.

Let $H<G_0$ be the fundamental group of a hyperplane $W$.  Since $W$ is convex in the (Gromov hyperbolic and CAT$(0)$) universal cover of $X_0$, $H$ is a quasiconvex subgroup.  
Using QCERF we can pass to a finite cover in which 
every elevation of $W$ is nonseparating.  (For the idea, see \cite[Lemma 2.2]{Lubotzky96}.)
There is therefore a finite regular cover $X_1\to X_0$ in which every hyperplane is nonseparating.  

On the level of
fundamental groups we have $G_1 = \pi_1X_1\dotnorm G_0$ with an induced peripheral structure 
$(G_1,\mc{P}_1)\dotnorm (G_0,\mc{P}_0)$ (again as in Notation \ref{subpair}).
We can suppose further that
$(G_1,\mc{P}_1)\dotnorm (G,\mc{P})$ by passing to a further finite
index subgroup if necessary.

\medskip\noindent\textbf{Step 2:} The subgroups $\mc{P}_1$ are quasiconvex in
$G_1$ so we can represent them by superconvex locally isometric immersions 
$\{Z(P)\stackrel{\phi_P}{\longrightarrow} X_1\mid P\in \mc{P}_1\}$ 
(Lemma \ref{lem:haglund}).  
Setting $\mc{Z}_1 = \bigcup_{\mc{P}_1} Z(P)$ we have a
superconvex cube complex pair $(X_1,\mc{Z}_1)$, and we can form the
associated augmented complex $C_1$.  The following easy lemma ensures that finite covers of $C_1$ are of the same form.
\begin{lemma}
  Any regular finite cover $C'$ of $C_1$ has the properties:
  \begin{enumerate}
  \item $C'$ is the augmented complex for a superconvex pair $(X',\mc{Z}')$ with 
  $$(G',\mc{P}')=\left(\pi_1(X'),\left\{\pi_1(Z)\mid Z\mbox{ a component of }\mc{Z}'\right\}\right)\dotnorm (G_1,\mc{P}_1)$$ and
  \item Every nonperipheral hyperplane of $C'$ is nonseparating.
  \end{enumerate}
\end{lemma}

\medskip\noindent\textbf{Step 3:}  
We fix some parameters.  Since $G$ is a hyperbolic group, the
universal cover $\tilde{C}$ of $C_1$ is $\delta$--hyperbolic for some
$\delta$.  Since $(G,\mc{P})$ is relatively hyperbolic, so is
$(G_1,\mc{P}_1)$.  In particular $\mc{P}_1$ is a malnormal collection.
Proposition \ref{malnormalseparated} then implies that the family of elevations of components of
$\mc{Z}_1$ to $\tilde{C}$ is $f$--separated for some affine $f$.
We let $\epsilon_A$ and $\epsilon_B$ be constants which are sufficiently large for the Combination Theorem \ref{combotheorem} to work.  (For example we may take $\epsilon_B = 50\delta+1$ and $\epsilon_A = 2\max\{f(\epsilon_B+2\delta)+\delta+\epsilon_B,2\epsilon_B\}+1$.)

Fundamental groups of hyperplanes and of peripheral complexes are quasiconvex in $G$.
Using QCERF we may pass to a finite regular cover $C_2\to C_1$ associated to a pair $(X_2,\mc{Z}_2)$ so that:
\begin{enumerate}
\item Every non-peripheral hyperplane is $\epsilon_A$--embedded in $C_2$; and
\item every peripheral subcomplex is $\epsilon_B$--embedded in $C_2$.
\end{enumerate}

\medskip\noindent\textbf{Step 4:}  
Let $\ddot{G}_2 = \pi_1(\ddot{C}_2)$, and let 
\[\ddot{\mc{P}}_2 = \{ \pi_1\tilde{Z}<\pi_1(\ddot{C}_2)\mid \tilde{Z}\mbox{ an elevation of a component }Z\subseteq\mc{Z}_2\mbox{ to }\ddot{C}_2\}.\]
We now can show the double-dot hierarchy $\mc{H}$ on $\ddot{C}_2$ is faithful and quasiconvex using Theorem \ref{goodhierarchy}.\eqref{good:qc}.  Since the peripheral subcomplexes are never cut up, this hierarchy is $\ddot{\mc{P}}_2$--elliptic.
  Since the pair $(X_2,\mc{Z}_2)$ is superconvex, Theorem \ref{goodhierarchy}.\eqref{good:noaccidents} implies that the edge spaces of the hierarchy have no accidental $\ddot{\mc{Z}}_2$--loops.  This implies that the hierarchy is \emph{fully} $\ddot{\mc{P}}_2$--elliptic.

The cube complex $C_2$ is the augmented complex associated to a pair $(X_2,\mc{Z}_2)$ so that every hyperplane of $X_2$ is embedded, $2$--sided, and nonseparating (from Step 1), so Corollary \ref{cor:terminalgroups} applies.  We deduce that the hierarchy of groups on $\ddot{G}_2$ terminates in free products of free groups and elements of $\ddot{\mc{P}}_2$.

\medskip\noindent\textbf{Step 5:}  
Pass to a further finite index subgroup $(G_3,\mc{P}_3)\dotnorm (\ddot{G}_2,\ddot{\mc{P}}_2)$ so the hierarchy induced by $\mc{H}$ is a malnormal hierarchy.  This can be done by Theorem \ref{t:vmalnormal}.

The terminal groups of the induced hierarchy are finite index in the terminal groups of $\mc{H}$.  In particular, they are
free products of free groups and elements of $\mc{P}_3$.  We can then continue the hierarchy (quasiconvexly, malnormally, and fully $\mc{P}_3$--elliptically) to one which terminates in $\mc{P}_3$.
\qed

\section{The (Malnormal) Special Combination Theorem}\label{s:SCT}
The Malnormal Quasiconvex Hierarchy Theorem of Wise is essential to our proof of Theorem \ref{thm:hierarchydescends}.  We explain how the result follows from the work in \cite{HsuWise15} and \cite{HaglundWise12}.
\begin{theorem}[Malnormal Quasiconvex Hierarchy]\label{MQH} \cite[Theorem 11.2]{Wise}
  If $G$ is a hyperbolic group with a malnormal quasiconvex hierarchy terminating in a collection of virtually special groups, then $G$ is virtually special.
\end{theorem}

Theorem \ref{MQH} follows from the following (malnormal version of the) Special Combination Theorem, by induction on the length of the hierarchy.  

\begin{theorem}[Malnormal Special Combination]\label{sct} \cite{HsuWise15} $+$ \cite{HaglundWise12}
  Suppose that $G$ is a hyperbolic group and that $G = \pi_1(\Gamma,\mc{G})$ for some faithful graph of groups $(\Gamma,\mc{G})$ where
\begin{enumerate}
\item\label{eq:malnormal qc} The edge groups of $(\Gamma,\mc{G})$ are malnormal and quasiconvex in $G$; and
\item The vertex groups of $(\Gamma,\mc{G})$ are virtually special.
\end{enumerate}
Then $G$ is virtually special.
\end{theorem}
\begin{remark}
In Section \ref{s:QCH} below we remark how the Malnormal Special Quotient Theorem and Dehn filling can be used to remove the condition of malnormality from \eqref{eq:malnormal qc} above in order to obtain another, much more powerful result of Wise from \cite{Wise}.
\end{remark}

\begin{proof}[Proof of Theorem \ref{sct}]
The result follows from the case that $\Gamma$ has a single edge by induction on the number of edges in $\Gamma$.  We therefore assume that $\Gamma$ is a one-edge splitting.

\cite[Main Theorem 8.1]{HsuWise15} implies that $G$ acts properly
 and cocompactly on a CAT$(0)$ cube complex $X$.

\setcounter{case}{0}

  \begin{case}
   $G = A*_C B$.
  \end{case}

 Since $C$ is quasiconvex, there is a $C$--cocompact convex
  subcomplex $X(C)\subseteq X$ \cite[Theorem 2.28]{Haglund08}.  Denote the quotient by the $C$--action $Y(C) = C\backslash X(C)$.
  Since $A$ and $B$ are quasiconvex, there are similarly an $A$--cocompact subcomplex $X(A)$ and a $B$--cocompact subcomplex $X(B)$ with quotients $Y(A) = A\backslash X(A)$ and $Y(B) = B\backslash X(B)$.  By enlarging $X(A)$ and $X(B)$ if necessary, we may assume their intersection contains $X(C)$.

  The inclusions of $X(C)$ into $X(A)$ and $X(B)$ induce local isometries
  $f\co Y(C) \to Y(A)$ and $g\co Y(C)\to Y(B)$.
  
  These give the necessary data to build a graph of spaces $Z$ with
  underlying graph a single edge.  It is easily verified that $Z$ is a
  non-positively curved cube complex, and the edge space $Y(C)$ is a separating
  hyperplane.  Since $A$ and $B$ are assumed to be virtually special
  groups, the vertex spaces $Y(A)$ and $Y(B)$ must be virtually
  special cube complexes, as explained in Remark \ref{remark:vspecialgroup}.
  Since $\pi_1(Y(C)) = C$ is malnormal in
  $G$, we may apply \cite[Theorem 8.5]{HaglundWise12} to conclude that $Z$ is
  virtually special.
  
\begin{case}
   $G = A*_C$.
\end{case}
  This case is similar, but the graph of spaces construction is slightly different.
  We suppose that $C<A$, and choose $t\in G$, $\phi\co C\to A$ so that $G$ has the presentation:
\[ \langle A,t \mid c = t\phi(c)t^{-1}\mbox{, }c\in C\rangle.\]
  Let $C' = \phi(C)$.  Choose some $C$--cocompact convex subcomplex
  $X(C)\subseteq X$.  Note that $X(C'):=t^{-1}X(C)$ is
  $C'$--cocompact.  
  Since $C$ and $C'$ are both subgroups of the
  quasiconvex subgroup $A$, we can choose an $A$--cocompact convex
  subcomplex $X(A)\subseteq X$ containing both $X(C)$ and $X(C')$.

  We now form the quotient spaces $Y(A) = A\backslash X(A)$,
  $Y(C)=C\backslash X(C)$, and $Y(C')= C'\backslash X(C')$.  Let
  $\iota\co Y(C)\to Y(A)$ and $\iota'\co Y(C')\to Y(A)$ be the maps
  induced by the inclusions in $X$.  The translation $x\mapsto
  t^{-1}x$ on $X$ restricts to a homeomorphism $X(C)\to X(C')$, which
  descends to a homeomorphism $\tau\co Y(C)\to Y(C')$.  We now have
  the data to build a graph of spaces with a single vertex space
  $Y(A)$ and a single edge space $Y(C)$.  The two maps of the edge
  space to $Y(A)$ are $\iota$ and $\iota'\circ \tau$, both local isometries.

  The space $Z$ thus constructed is a non-positively curved cube complex whose
  fundamental group is isomorphic to $G$.  Moreover, it contains a
  hyperplane $H$ with fundamental group equal to $C$, and
  $Z\smallsetminus N(H)$ is virtually special.  We may apply
  \cite[Theorem 8.5]{HaglundWise12} to conclude that $Z$ is virtually special.
\end{proof}
\begin{remark}
  The preceding argument does not use the assumption that the edge
  groups of the hierarchy are malnormal in $G$, but only the ``local''
  information that they are malnormal in the next level up.
\end{remark}

\section{Relative hyperbolicity and Dehn filling} \label{s:RHDF}
In this section we state the group theoretic Dehn filling results
needed to prove Theorem \ref{thm:hierarchydescends}.
\subsection{Group theoretic Dehn filling}
We first recall the definition of group theoretic Dehn filling (Definition \ref{def:dehnfill}).
From a group pair $(G,\mc{P})$ and a collection 
$\{N_i\lhd P_i\mid P_i\in \mc{P}\}$,
we obtain a quotient of $G$ and a filling map
\begin{equation}\label{eq:filling}
\pi \co G \to G(N_1,\ldots,N_m) = G/\llangle \cup_i N_i\rrangle .
\end{equation}

\begin{definition}
  Let $(G,\mc{P})$ be relatively hyperbolic.
  A statement is true \emph{for all sufficiently long fillings} if
  there is a finite set $B\subset G\smallsetminus\{1\}$ so that the
  statement holds for all fillings as in \eqref{eq:filling}
  with
  $B\cap(\cup_i N_i)=\emptyset$.
\end{definition}

The basic group theoretic Dehn filling result is the following (but see \cite{DGO} for an interesting generalization).
\begin{theorem}[Relatively hyperbolic Dehn filling \cite{osin:peripheral}, cf. \cite{rhds}] \label{thm:RHDF}
Let $G$ be a group and $\mc{P} = \{ P_1, \ldots , P_m \}$ a collection of subgroups so that $(G,\mc{P})$ is relatively hyperbolic.  Let $F \subset G$ be finite.  For all sufficiently long fillings
\[	\phi \co G \to \bar{G} := G(N_1,\ldots,N_m)	;	\]
\begin{enumerate}
\item $\ker(\phi|_{P_i}) = N_i$ for $i = 1,\ldots,m$;
\item $(\bar{G},\{ \phi(P_1),\ldots , \phi(P_m))$ is relatively hyperbolic; and
\item $\phi|_F$ is injective.
\end{enumerate}
\end{theorem}
 
\subsection{Quasi-convex Dehn filling results}
In this subsection we recall the definition of relatively quasiconvex subgroup and the statements relating quasiconvexity to Dehn fillings.  
\begin{definition}
Let $(G,\mc{P})$ be a relatively hyperbolic group (relatively) generated by a
finite subset $S\subseteq G$.  The \emph{relative Cayley graph}
$\hat{\Gamma}=\Gamma(G,S\sqcup (\bigcup \mc{P}))$ is then a $\delta$--hyperbolic, fine graph \cite{bowditch:relhyp,osin:relhypbook}.
A subgroup $H<G$ is said to be \emph{relatively quasiconvex} in $(G,\mc{P})$ if there is a constant $\lambda$ so that whenever $v\in G$ lies on a $\hat{\Gamma}$--geodesic with endpoints in $H$, then $d_S(v,H)\leq \lambda$.
\end{definition}

\begin{remark}  
There are many equivalent notions of relatively quasiconvex subgroups -- \cite{HruskaQC}(cf. \cite{HruskaQCerratum}), \cite{agm1}, \cite{MM-P}.
\end{remark}

The following result follows immediately from \cite[Theorem 1.5]{HruskaQC} or \cite[Theorem 1.1]{MartinezQC}.

\begin{proposition}\label{prop:QCisRQC}
Suppose that $G$ is hyperbolic, that $H$ is quasiconvex in $G$ and that $(G,\mc{P})$ is relatively hyperbolic.  Then $H$ is a relatively quasiconvex subgroup of $G$.
\end{proposition}

\begin{theorem}\label{relqcthenrelhyp} \cite{HruskaQC}
A relatively quasiconvex subgroup $H$ of a relatively hyperbolic group $(G,\mc{P})$  has a (finite) collection of peripheral subgroups $\mc{D}$ with respect to which it is relatively hyperbolic.  Moreover, the peripheral subgroups $\mc{D}$ of $H$ are conjugate into elements of $\mc{P}$.
\end{theorem}

\begin{remark}\label{rk:modifyD} 
  By modifying $\mc{D}$ if necessary, we can always suppose 
  \begin{enumerate}
  \item each element of $\mc{D}$ is
  infinite, and
  \item every infinite intersection of $H$ with a
  conjugate of some $P\in \mc{P}$ is conjugate in $H$ to an element of
  $\mc{D}$.
  \end{enumerate}
\end{remark}

\begin{definition}
Let $(G,\mc{P})$ be relatively hyperbolic.  A relatively quasiconvex subgroup $H$ of $G$ is {\em fully quasiconvex} if for any $P \in \mc{P}$ and any $g \in G$ the subgroup $H \cap P^g$ is either finite or else has finite-index in $P^g$.
\end{definition}

\begin{definition}[$H$--fillings] \label{def:hfilling}  
  Let $G$ be hyperbolic relative to $\mc{P}=\{P_i,\ldots,P_m\}$, and let $H<G$.  
  We say that a filling $G\to G(N_1,\ldots,N_m)$ is an
  \emph{$H$--filling} if, whenever $gP_ig^{-1}\cap H$ is infinite, for
  $P_i\in \mc{P}$ and $g\in G$, it follows that $gN_ig^{-1}\subseteq H$.  
\end{definition}
\begin{definition}[Induced filling]\label{def:inducedfilling}
  If $H$ is relatively quasiconvex in $(G,\mc{P})$ then Theorem
  \ref{relqcthenrelhyp} gives $H$ a peripheral structure $\mc{D}$ so
  that $(H,\mc{D})$ is relatively hyperbolic.  We suppose $\mc{D}$ satisfies the conditions of Remark \ref{rk:modifyD}.
  Consider an $H$--filling
\[ \pi \co G\to G(N_1,\ldots,N_m).\]
  Let $D_j\in \mc{D}$.
  There is some $P_i\in \mc{P}$ and some $g\in G$ with $g^{-1}D_j
  g\subseteq P_i$.  We define $K_j = g N_i g^{-1}$.  (Although $D_j$
  determines $P_i$ it doesn't quite determine $g$.  Note however that
  since $N_i$ is normal in $P_i$, the group $K_j$ is independent of the
  conjugating element $g$.)  Because $\pi$ is an $H$--filling, we have
  $K_j\lhd D_j$, so these groups determine a filling of $H$
\[ \pi_H\co H\to H(K_1,\ldots,K_n),\]
  called the \emph{induced filling of $H$}.
\end{definition}

The next theorem summarizes results from \cite{agm1}, where $G$ is
assumed to be torsion-free and a slightly different definition of
$H$--filling is used.  As explained in \cite[Appendix B]{MM-P} (cf. \cite[Appendix]{VH}), the torsion-free
assumption is unnecessary, so long as we use
Definition \ref{def:hfilling}.

\begin{theorem}\label{qcfilling}
  Let $(G,\mc{P})$ be relatively hyperbolic, let $H<G$ be fully relatively quasiconvex, and let $F\subset G$ be finite.  For all sufficiently long $H$--fillings $\phi\co G\to G(N_1,\ldots , N_m)$ of $G$:
  \begin{enumerate}
  \item \cite[Proposition 4.3]{agm1}\label{frq} $\phi(H)$ is fully relatively quasiconvex.
  \item \cite[Proposition 4.4]{agm1}\label{inducedfilling} $\phi(H)$ is isomorphic to the induced filling (see Definition \ref{def:inducedfilling}).  More precisely if $\phi_H\co H\to H(K_1,\ldots,K_n)$ is the induced filling map, then
$(\ker\phi)\cap H = \ker \phi_H$.
  \item \cite[Proposition 4.5]{agm1}\label{gnotinH} $\phi(F)\cap \phi(H) = \phi(F\cap H)$.
  \end{enumerate}
\end{theorem}

The next statement we need is immediate from the discussion preceding Corollary A.46 in \cite{VH}.  Note that this discussion does not rely in any way on the Malnormal Special Quotient Theorem, even though the main result of that appendix does.  Also note that in the statements of the Appendix to \cite{VH}, the assumption is made that the peripheral structure $\mc{P}$ is the structure induced by the quasiconvex group $H$.  This assumption is not used until we argue that height \emph{strictly} decreases after filling.  For the rest (including the following statement) it is only important that $H$ is fully quasiconvex in $(G,\mc{P})$. 
\begin{theorem} [Height reduction] \label{t:height reduction}
Suppose that $(G,\mc{P})$ is relatively hyperbolic.  Suppose that $H \le G$ is fully quasi-convex of height $k$.  For all sufficiently long $H$-fillings $G \to \bar{G}$, the image $\bar{H}$ of $H$ in $\bar{G}$ has height at most $k$.
\end{theorem}

\section{Induced splittings of quotients} \label{sec:inducedsplittings}
In this section we prove Theorem \ref{thm:hierarchydescends}, which shows that (with appropriate assumptions) nice hierarchies of $(G,\mc{P})$ descend to nice hierarchies of Dehn fillings of $(G,\mc{P})$.

\subsection{Definition of the filled hierarchy}
Let $(G,\mc{P})$ be relatively hyperbolic.  Let $\mc{H}$ be a quasiconvex fully $\mc{P}$--elliptic hierarchy of $G$.  In this setting, every vertex or edge group of $\mc{H}$ is fully relatively quasiconvex.  Let $\pi\co (G,\mc{P})\to (\bar{G},\bar{\mc{P}})$ be a Dehn filling of $(G,\mc{P})$.  In this subsection we describe a (not necessarily faithful) hierarchy on $(\bar{G},\bar{\mc{P}})$.
We call this hierarchy $\bar{\mc H}$
the \emph{filled hierarchy.}

Each of the groups in the filled hierarchy is the induced filling (in
the sense of Definition \ref{def:hfilling}) of
some group in the hierarchy $\mc{H}$.

At level $0$, the filled hierarchy $\bar{\mc H}$ consists of the degenerate graph
of groups structure on $\bar{G}$.  Suppose we have defined the filled
hierarchy down to level $i$ less than the length of the hierarchy $\mc{H}$, and let $\bar{A}\in \bar{\mc{H}}$ be
a vertex group at level $i$.  The group $\bar{A}$ is the induced
filling of a vertex group $A$ at level $i$ of $\mc{H}$.
This group
$A$ comes equipped with a graph of groups structure $\alpha\co
\pi_1(\Gamma,\mc{A},a_0)\to A$, from which we must build a graph of
groups structure on $\bar{A}$.  This graph of groups structure
has the exact same underlying graph $\Gamma$ and base vertex $a_0\in\Gamma$, and differs only in the assignment $\bar{\mc{A}}$ of
groups and homomorphisms, which is defined as follows: If $x$ is a
vertex or edge of $\Gamma$, then $\bar{\mc A}_x$ is defined to be the
induced filling of $\mc{A}_x$.  
Write $\pi_x\co \mc{A}_x\to \bar{\mc A}_x$ for the induced filling map.  
Let $e$ be an edge of $\Gamma$ with $t(e)=v$, and let
$\phi_e\co \mc{A}_e\to \mc{A}_v$ be the homomorphism coming from the graph of groups.  We have to fill in the square:
\cd{\mc{A}_e\ar[r]^{\pi_e}\ar[d]_{\phi_e} & \bar{\mc{A}}_{e}\ar@{-->}[d]\\
\mc{A}_{v} \ar[r]_{\pi_{v}} & \bar{\mc A}_{v} }
The necessary condition $\ker\pi_e \subseteq \ker (\pi_{v}\circ\phi_e)$ holds, so we get an induced map $\bar{\phi}_e\co \bar{\mc A}_e\to \bar{\mc A}_{v}$.  These induced maps complete the definition of the graph of groups $(\Gamma,\bar{\mc A})$. 

We emphasize that the edge maps for this graph of groups may not be injective.  In our applications (eg Corollary \ref{c:filledhierarchy}) we must establish injectivity.

To complete the construction of  $\bar{\mc{H}}$ we need an isomorphism $\bar{\alpha}\co \pi_1(\Gamma,\bar{\mc A},a_0)\to \bar{A}$.
Coming from $\mc{H}$ we have an isomorphism $\alpha\co \pi_1(\Gamma,\mc{A},a_0)\to A$.  
Also, the induced fillings $\pi_x\co \mc{A}_x\to \bar{\mc A}_x$ commute with the edge maps, so they
suffice to define a surjection $\pi_\Gamma\co \pi_1(\Gamma,\mc{A},a_0)\to \pi_1(\Gamma,\bar{\mc A},a_0)$.
To define $\bar{\alpha}$ we need to complete the square
\cdlabel{alphabar}{
\pi_1(\Gamma,\mc{A},a_0)\ar[r]^{\pi_{\Gamma}}\ar[d]_\alpha & \pi_1(\Gamma,\bar{\mc A},a_0)\ar@{-->}[d]\\
A\ar[r]_{\pi_A} & \bar{A},
}
where the map at the bottom is the induced filling map of $A$.
\begin{lemma}\label{lem:splittingoffilling}
  The diagram \eqref{alphabar} can be filled in with an isomorphism $\bar{\alpha}$.
\end{lemma}
\begin{proof}
  To complete the diagram with an isomorphism $\bar\alpha$
  it is necessary only to show that $\ker \pi_\Gamma=\ker(\pi_A\circ\alpha)$.

  We first show $\ker \pi_\Gamma\subseteq\ker(\pi_A\circ\alpha)$.
Let $k\in \ker \pi_\Gamma$.  
Then there is an expression 
\[k = \prod_i {k_i}^{g_i}\] 
where each $g_i\in \pi_1(\Gamma,\mc{A},a_0)$, and each $k_i$ lies in the kernel of some $\pi_{v_i}$ where $v_i$ is a vertex of $\Gamma$.  (We assume we have made some choice of maximal tree in $\Gamma$ so we can regard $\mc{A}_v$ as a subgroup of $\pi_1(\Gamma,\mc{A},a_0)$.)
To show $\ker \pi_\Gamma\subseteq \ker(\pi_A\circ\alpha)$ it therefore suffices to show $\ker \pi_v\subseteq \ker \pi_A\circ\alpha$ for each $v$.
We therefore assume without loss of generality that $k \in \ker\pi_v$ for some vertex $v$.  

The hierarchy is assumed quasiconvex and fully $\mc{P}$--elliptic. Thus $\mc{A}_v$ is fully relatively quasiconvex in $(G,\mc{P})$.
Even better, there is an induced peripheral structure $\mc{D}_v=\{D_1,\ldots,D_n\}$ on $\mc{A}_v$ so that for each $i$, there is an infinite $P_{j_i}\in \mc{P}$ and a $g_i\in G$ so that $D_i = (P_{j_i})^{g_i}$.  
Since $\pi_v \co \mc{A}_v\to \bar{\mc A}_v$ is the filling induced by $\pi$, the element $k$ has an expression
\[ k = \prod_\beta {n_\beta}^{a_\beta} \]
with each $n_\beta$ in some induced filling kernel $N_\beta\subseteq D_\beta$ and each $a_\beta\in \mc{A}_v$.  It suffices then to show each of these $n_\beta$ is in $\ker \pi_A\circ\alpha$.  But this is clear, since the conjugate $(P_{j_\beta})^{g_\beta}$ must be conjugate in $A$ to some element of the induced peripheral structure on $A$, and so $n_\beta$ is conjugate into some filling kernel of the induced filling $\pi_A$ on $A$.

Conversely, if $k\in \ker(\pi_A\circ\alpha)$, then $k = \prod_{i=1}^l \alpha^{-1}(k_i^{g_i})$ where each $k_i$ is in some filling kernel $K_{j_i}\lhd D_{j_i}$ for the induced filling $\pi_A$.
But since the hierarchy is fully $\mc{P}$--elliptic, 
the subgroup $\alpha^{-1}(D_{j_i})$ is conjugate into some vertex group $\mc{A}_v$
of $(\Gamma,\mc{A})$.  It follows that $\alpha^{-1}(k_i)$ is conjugate into $\ker \pi_v$ for some $v\in \Gamma$, so $\alpha^{-1}(k_i)$ is in $\ker \pi_\Gamma$.
\end{proof}

\subsection{Proof of Theorem \ref{thm:hierarchydescends}} \label{sec:hierdescends}

The first lemma is a straightforward application of the definitions.
\begin{lemma}
  Let $G$ be hyperbolic, and $\mc{P}$ be a family of subgroups so that $(G,\mc{P})$ is relatively hyperbolic.  Let $\mc{H}$ be a quasiconvex fully $\mc{P}$--elliptic hierarchy of $G$, and let $A<G$ be an edge or vertex group of $\mc{H}$.  Then $A$ is fully relatively quasiconvex in $(G,\mc{P})$, and every filling of $(G,\mc{P})$ is an $A$--filling.
\end{lemma}

\begin{lemma}\label{quasiconvexedges}
  Let $G$ be hyperbolic, and $\mc{P}$ be a family of subgroups so that $(G,\mc{P})$ is relatively hyperbolic.  Let $\mc{H}$ be a quasiconvex fully $\mc{P}$--elliptic hierarchy of $G$, and let $A$ be an edge or vertex group of $\mc{H}$.  Then for all sufficiently long fillings \[ \phi \co (G,\mc{P}) \to (\bar{G},\bar{\mc{P}}), \] 
  the following hold whenever $A$ is a vertex or edge group of $\mc{H}$:
  \begin{enumerate}
    \item\label{num:frq} The subgroup $\bar{A}=\phi(A)$ is fully relatively quasiconvex in $(\bar{G},\bar{\mc{P}})$.
    \item\label{num:qc} If $\bar{G}$ is hyperbolic, then $\bar{A}$ is quasiconvex in $\bar{G}$.
    \item\label{num:induced} The subgroup $\bar{A}$ is isomorphic to the induced filling of $A$.
    \item\label{num:height} The height of $\bar{A}$ in $\bar{G}$ is at most the height of $A$ in $G$.
  \end{enumerate}
\end{lemma}
\begin{proof}
  As there are only finitely many edge and vertex groups occurring in $\mc{H}$, statements \eqref{num:frq}, \eqref{num:induced}, and \eqref{num:height} follow from \ref{qcfilling}.\eqref{frq}, \ref{qcfilling}.\eqref{inducedfilling} (wih $F=\emptyset$), and \ref{t:height reduction}, respectively.  The statement \eqref{num:qc} follows from \eqref{num:frq} by \cite[Theorem 10.5]{HruskaQC} (which implies that in this case $\bar{A}$ is undistorted in $\bar{G}$) and \cite[Corollary III.$\Gamma$.3.6]{bridhaef:book}.
\end{proof}

In particular, Lemma \ref{quasiconvexedges}.\eqref{num:induced} has the immediate consequence that the filled hierarchy is faithful:
\begin{corollary}\label{c:filledhierarchy}
  Let $G$ be hyperbolic, and $\mc{P}$ be a family of subgroups so that $(G,\mc{P})$ is relatively hyperbolic.  Let $\mc{H}$ be a quasiconvex fully $\mc{P}$--elliptic hierarchy of $G$.  Then for all sufficiently long fillings  \[ \pi \co (G,\mc{P}) \to (\bar{G},\bar{\mc{P}}), \]
  the filled hierarchy $\bar{\mc H}$ on $\bar{G}$ is faithful.
\end{corollary}

\begin{proof}[Proof of Theorem \ref{thm:hierarchydescends}]
As in the statement, suppose that $(G,\mc{P})$ is a relatively hyperbolic pair with $G$ hyperbolic, and suppose $G$ is equipped with a malnormal  quasiconvex fully $\mc{P}$--elliptic hierarchy $\mc{H}$ terminating in $\mc{P}$.  
Let 
\[ \pi \co (G,\mc{P}) \to (\bar{G},\bar{\mc P}) \]
be a filling which is sufficiently long with respect to Lemma \ref{quasiconvexedges} and Corollary \ref{c:filledhierarchy}. By Corollary \ref{c:filledhierarchy}, the quotient $\bar{G}$ inherits a hierarchy $\bar{\mc H}$ with the same underlying graphs and with each of the edge or vertex groups equal to the image under $\pi$ of a corresponding edge or vertex group of $\mc{H}$.  In particular, all these edge and vertex groups are quasiconvex, by Lemma \ref{quasiconvexedges}.\eqref{num:qc}, so $\bar{\mc H}$ is a quasiconvex hierarchy.  By \ref{quasiconvexedges}.\eqref{num:height}, the edge groups of $\bar{\mc H}$ are all malnormal in $\bar{G}$, so $\bar{\mc H}$ is a malnormal hierarchy.  The construction of the filled hierarchy ensures the terminal groups are the induced fillings of the terminal groups of $\mc{H}$.  Since the terminal groups of $\mc{H}$ are in $\mc{P}$, the terminal groups of $\bar{\mc H}$ are in $\bar{\mc P}$.
\end{proof}

\section{Proof of the main theorem}\label{sec:mainthm}

Before we prove the main theorem, we introduce some terminology.

\begin{definition} \label{def:equivchosen}
Suppose that $(G',\mc{P'}) \dotnorm (G,\mc{P})$, and that $\{ N_j' \lhd P_j' \mid P_j'\in\mc{P'} \}$ is a collection of filling kernels.

We say that $\{ N_j' \}$ is {\em equivariantly chosen }if (i) whenever $gP_j'g^{-1}$
  and $hP_k'h^{-1}$ both lie in $P_i$, then $gN_j'g^{-1}=hN_k'h^{-1}$; and (ii) Each such $gN_j'g^{-1}$ is normal in the $P_i$ which contains it.
  
  An {\em equivariant filling} of $(G',\mc{P'})$ is a filling with an equivariantly chosen set of filling kernels.
\end{definition}

The following observation is the reason for considering equivariant fillings.

\begin{proposition} \label{filling of original G}
An equivariant filling $(G',\mc{P'}) \to (\bar{G}',\bar{\mc{P}'})$ determines a filling $(G,\mc{P}) \to (\bar{G},\bar{\mc{P}})$, so that $(\bar{G}',\bar{\mc P}')\dotnorm (\bar{G}, \bar{\mc P})$.
\end{proposition}

We now apply Wise's Malnormal Quasiconvex Hierarchy Theorem (Theorem \ref{MQH}), together with our  Theorems \ref{thm:virtualhierarchy} and \ref{thm:hierarchydescends}, to establish the main result.
\begin{proof}[Proof of Theorem \ref{msqt+}]
  Let $(G,\mc{P})$ satisfy the hypotheses of Theorem \ref{msqt+}.
  Theorem \ref{thm:virtualhierarchy} implies that there is a subgroup
  $(G',\mc{P}')\dotnorm (G,\mc{P})$ with a malnormal quasiconvex fully $\mc{P}'$-elliptic hierarchy terminating in $\mc{P}'$.

  According to Theorem \ref{thm:hierarchydescends}, any sufficiently long filling $(\bar{G}',\bar{\mc{P}}')$ of $(G',\mc{P}')$ has a quasiconvex malnormal fully $\bar{\mc{P}}'$-elliptic hierarchy terminating in $\bar{\mc{P}}'$.  Since $G$ is virtually special, it is residually finite, so there exist peripherally finite fillings of $(G',\mc{P'})$ which are sufficiently long for the conclusions of Theorem \ref{thm:hierarchydescends} to hold.
Let $G'(K_1,\ldots,K_M)$ be some such peripherally finite filling, which we may assume is also sufficiently long for the conclusions of Theorem \ref{thm:RHDF} to hold.  
We then modify it to be equivariant with respect to $G$, by setting
\[ K_i' = \bigcap\left\{K_j^g\mid g\in G,\#( K_j^g\cap P_i)=\infty\right\}.\]
We note that these modified filling kernels are contained in the old ones, so the modified filling $G'(K_1',\ldots,K_M')$ will still be sufficiently long for the conclusions of Theorems \ref{thm:hierarchydescends} and \ref{thm:RHDF} to hold.

The filling $G'(K_1',\ldots,K_M')$ is still peripherally finite, and by Proposition \ref{filling of original G} it moreover determines a filling of $G$.  For $P_j\in \mc{P}$, define $\dot{P}_j$ to be equal to $(K_i')^g$ for some (any) $i$, $g$ so that $(K_i')^g\subset P_j$.  The equivariance ensures there is no ambiguity in this definition.

Now consider any filling $G \to G(N_1,\ldots,N_m)$ so that for each $i$ we have (i) $N_i\lhd P_i$; (ii)  $N_i<\dot{P}_i$; and (iii) $P_i/N_i$ is virtually special and hyperbolic.  There is an induced equivariant filling
\[ G' \longrightarrow G'(N_1',\ldots, N_M') \]
so that $N_j'< K_j'$ and $N_j'\lhd P_j'$ for each $j$.  This filling
is sufficiently long that the conclusion of Theorem \ref{thm:hierarchydescends} holds.  In particular the
filling $G'(N_1',\ldots, N_M')$ has a malnormal quasiconvex hierarchy
terminating in $\bar{\mc{P}}' = \{ P_j'/N_j'\}$.  By Theorem \ref{thm:RHDF}, the pair $(G'(N_1',\ldots , N_M'), \bar{\mc{P}}')$ is relatively hyperbolic; since the groups in $\bar{\mc P}'$ are hyperbolic, so is the filling.  Thus $G'(N_1', \ldots, N_M')$ is a hyperbolic group with a malnormal quasiconvex hierarchy terminating in a collection of hyperbolic, virtually special groups.
By Theorem \ref{MQH}, $G'(N_1', \ldots, N_M')$ is virtually special.

By Proposition \ref{filling of original G}, 
$G'(N_1',\ldots, N_M')\dotnorm G(N_1,\ldots,N_m)$, so the filling $G(N_1,\ldots,N_m)$ of the original group is also virtually special.
\end{proof}

\section{Quasi-convex virtual hierarchies} \label{s:QCH}

In \cite{Wise}, Wise proves a theorem about groups with {\em virtual} hierarchies, which are slight generalizations of hierarchies allowing finite-index subgroups along the way.  We recall the definition and statement of Wise's Quasi-convex Hierarchy Theorem.

\begin{definition} \cite[Definition 11.5]{Wise}
Let $\mc{QVH}$ denote the smallest class of hyperbolic groups that is closed under the following operations.
\begin{enumerate}
\item $\{1 \} \in \QVH$.
\item If $G = A \ast_B C$ and $A,C \in \QVH$ and $B$ is finitely generated and quasi-isometrically embedded in $G$ then $G \in \QVH$.
\item If $G = A \ast_B$ and $A \in QVH$ and $B$ is finitely generated and quasi-isometrically embedded then $G \in \QVH$.
\item If $H < G$ with $|G:H| < \infty$ and $H \in \QVH$ then $G \in \QVH$.
\end{enumerate}
\end{definition}

\begin{theorem} \cite[Theorem 13.3]{Wise} \label{t:QVH}
A hyperbolic group is in $\QVH$ if and only if it is virtually (compact) special.
\end{theorem}
This is stated as \cite[Theorem 2.9]{VH} (with the extra assumption that the group is torsion-free).  Other than the Malnormal Special Quotient Theorem, it is the result from \cite{Wise} that is used in \cite{VH}.  It is used in \cite[Theorem 3.1]{VH}.

In this section we finish by showing how to deduce Wise's Quasiconvex
Hierarchy Theorem from the Malnormal Special Quotient Theorem and some
Dehn filling results from the Appendix of \cite{VH}.

\subsection{Induced peripheral structure from a quasiconvex subgroup}
We first briefly recall a construction from \cite{agm1,VH}.  Let $G$
be a hyperbolic group, and let $H$ be a quasiconvex subgroup.  We describe peripheral structures $\mc{D}$ on $H$ and $\mc{P}$ on $G$ so that
\begin{enumerate}
\item $(H,\mc{D})$ and $(G,\mc{P})$ are relatively hyperbolic;
\item $H$ is fully relatively quasiconvex in $(G,\mc{P})$ and $\mc{D}$ satisfies the conditions given in Remark \ref{rk:modifyD}; and
\item every element of $\mc{P}$ is commensurable in $G$ with at least one element of $\mc{D}$.
\end{enumerate}
We refer to \cite{agm1,VH} for proofs of the above facts.

By \cite[Corollary 3.5]{agm1}, there are only finitely many
$H$--conjugacy classes of minimal infinite subgroups of the form
\begin{equation}\label{Hintersection}
H\cap g_2Hg_2^{-1}\cap\cdots\cap g_jHg_j^{-1}.
\end{equation} 
The \emph{malnormal core} $\mc{D}$ of $H$ is formed as follows:
\begin{enumerate}
\item Let $\mc{D}''$ be the set of minimal infinite subgroups of the form \eqref{Hintersection}.
\item Form $\mc{D}'$ from $\mc{D}''$ by replacing each $D\in \mc{D}''$ with its $H$--commensurator.
\item Form $\mc{D}$ from $\mc{D}'$ by choosing one representative of each $H$--conjugacy class.
\end{enumerate} 
The \emph{induced peripheral structure} $\mc{P}$ in $G$ is then formed from $\mc{D}$ via two similar steps:
\begin{enumerate}
\item Form $\mc{P}'$ from $\mc{D}$ by replacing each $D\in \mc{D}$ with its $G$--commensurator.
\item Form $\mc{P}$ from $\mc{P}'$ by choosing one representative of each $G$--conjugacy class.
\end{enumerate}

In this section, we use the following refinement of Theorem \ref{t:height reduction} from \cite{VH}.  Note that the proof of this result is elementary and does not use the Malnormal Special Quotient Theorem or indeed any results about cube complexes.
\begin{theorem}\label{t:strictheightreduce}
  \cite[Theorem A.47]{VH}  Let $H$ be a quasiconvex subgroup of the hyperbolic group $G$, let $\mc{D}$ be the malnormal core of $H$ in $G$ and let $\mc{P}$ be the peripheral structure induced on $G$ by $H$.  For all sufficiently long peripherally finite fillings $\phi\co G\to \bar{G} = G(N_1,\ldots,N_m)$, the height of $\phi(H)$ in $\bar{G}$ is strictly less than that of $H$ in $G$.
\end{theorem}

\subsection{Induced peripheral structure from a one-edge splitting}
In this subsection we suppose that $G$ is a hyperbolic group, and that $G$ has a one-edge splitting with edge group $H<G$.  
That is, either $G = A\ast_H B$ or $G = A\ast_\phi$, where $\phi\co H\to H'$ is an isomorphism of two subgroups of $G$.  We suppose further that the edge group $H$ is quasiconvex.

To the splitting corresponds a Bass-Serre tree on which $G$ acts, the edges of which are in one-to-one correspondence with left cosets of $H$ in $G$.  Let $e$ be the edge corresponding to the trivial coset.  Let $\mc{T}$ be the collection of maximal subtrees of $T$ with infinite stabilizer, and let $\mc{T}_e = \{ T\in \mc{T} \mid e\subseteq T\}$.  Because the height of $H$ in $G$ is finite, each tree in $\mc{T}$ is finite.
\begin{observation}
  With the above notation:
  \begin{enumerate}
  \item Each $D\in \mc{D}$ preserves some tree in $\mc{T}_e$; this gives a one-to-one correspondence of $\mc{D}$ with the orbit space $H \backslash \mc{T}_e$.
  \item Each $P\in \mc{P}$ also preserves some tree in $\mc{T}_e\subset \mc{T}$, giving a one-to-one correspondence of $\mc{P}$ with the orbit space $G \backslash \mc{T}$.
  \item Since each tree $T\in \mc{T}$ is finite, each conjugate $g Pg^{-1}$ with $P\in \mc{P}$ has a finite index subgroup which is equal to the stabilizer of a tree in $\mc{T}$.
  \end{enumerate}
\end{observation}

\begin{lemma}\label{vertexgroupsarefull}
  Let $G$ be hyperbolic, with a one-edge splitting over a quasiconvex subgroup $H$, and with $\mc{D}$, $\mc{P}$ as above.  The vertex groups of the splitting are fully relatively quasiconvex in $(G,\mc{P})$.  
\end{lemma}
\begin{proof}
  Let $P$ be a conjugate of an element of $\mc{P}$ which intersects a vertex group $A$ in an infinite subgroup.  A finite index subgroup $P'<P$ is the stabilizer of some tree $T\in \mc{T}$.  The vertex group $A$ is the stabilizer of some vertex $v$ in the Bass-Serre tree.  Since $A\cap P$ is infinite, so also is $A\cap P'$.  This infinite subgroup $A\cap P'$ therefore stabilizes all of $T\cup \{v\}$.  It follows that $v\in T$, since $T$ is a maximal tree with infinite stabilizer.  Thus $P'$ fixes $v$, and so $P'< A$.
\end{proof}

\subsection{Virtual torsion-freeness}
In this subsection we show that a group with a quasiconvex one-edge splitting and virtually special vertex groups is virtually torsion-free.

\begin{lemma} \label{l:vtf}
Suppose that $(\Gamma,\mc{G})$ is a (faithful) finite graph of groups with finite edge groups and virtually torsion-free vertex groups.  Then $\pi_1(\Gamma,\mc{G})$ is virtually torsion-free.
\end{lemma}
\begin{proof}
  Each vertex group $\mc{G}_v$ contains some finite index normal $N_v$ which is torsion-free.  We can form a new graph of groups $(\Gamma,\mc{G}')$ with the same underlying graph of groups $\Gamma$ and the same edge groups, but with $\mc{G}'_v = \mc{G}_v/N_v$.  Since the kernels of the maps $\mc{G}_v\to \mc{G}'_v$ are torsion-free, each edge inclusion map $\phi_e\co \mc{G}_e\to \mc{G}_{t(e)}$ induces a well defined monomorphism $\phi_e'\co \mc{G}'(e) \to \mc{G}'_{t(e)}$.

  Because the vertex groups of $(\Gamma, \mc{G})$ are quotients of the vertex groups of $(\Gamma,\mc{G}')$ and the
  edge-inclusion maps of $(\Gamma, \mc{G}')$ are induced by those of $(\Gamma,\mc{G})$, there is an induced surjection
\[ q\co \pi_1(\Gamma,\mc{G})\to \pi_1(\Gamma,\mc{G}').\]
Since the maps $\mc{G}_v\to \mc{G}_v'$ have torsion-free kernel, so does $q$.  But $\pi_1(\Gamma,\mc{G}')$ is a graph of finite groups, so it is virtually free.  Let $F\dotnorm \pi_1(\Gamma,\mc{G}')$ be free.  Then $q^{-1}(F)\dotnorm \pi_1(\Gamma,\mc{G})$ is torsion-free.
\end{proof}

\begin{proposition} \label{p:qvh is vtf}
Suppose the hyperbolic group $G$ admits a one-edge splitting with edge group equal to the quasiconvex subgroup $H$.  Suppose the vertex groups of this splitting are virtually special.  Then $G$ is virtually torsion-free.
\end{proposition}
\begin{proof}
We argue by induction on the height $n$ of $H$ in $G$.  If $n=0$, then the result follows from Lemma \ref{l:vtf}, since virtually special groups are virtually torsion-free.

If $n>0$, then let $\mc{D}$ be the malnormal core of $H$, and let $\mc{P}$ be the peripheral structure on $G$ induced by $H$.  By Lemma \ref{vertexgroupsarefull}, the vertex groups of the splitting over $H$ are fully relatively quasiconvex.  
The elements of $\mc{P}$ are virtually quasiconvex subgroups of vertex groups, so they are virtually special.  Therefore there exist arbitrarily long peripherally finite fillings of $(G,\mc{P})$.  Since the vertex groups are fully relatively quasiconvex, there are even arbitrarily long peripherally finite fillings which are both $H$--fillings and $A$--fillings for every vertex group $A$.  

By the Malnormal Special Quotient Theorem (Corollary \ref{msqt}), there are also arbitrarily long peripherally finite such fillings so that the induced fillings of the vertex groups are themselves virtually special.
Choose some such peripherally finite filling $\phi\co G\to \bar{G}=G(N_1,\ldots,N_m)$, long enough so that:
\begin{enumerate}
\item $\phi$ has no nontrivial torsion elements in its kernel (this is possible because there are only finitely many conjugacy classes of torsion elements in a hyperbolic group),
\item $\phi$ restricts to the induced filling on $H$ and on each vertex group (see  \ref{qcfilling}.\eqref{inducedfilling} with $F=\emptyset$),
\item $\bar{G}$ is hyperbolic and $\bar{H} = \phi(H)$ is quasiconvex in $\bar{G}$ (see \ref{qcfilling}.\eqref{frq} with $F=\emptyset$),
\item the height of $\phi(H)$ in $\bar{G}$ is strictly less than the height of $H$ in $G$ (Theorem \ref{t:strictheightreduce}).
\end{enumerate}

Because the filling $\phi$ is an $H$--filling as well as an $A$--filling for each vertex group $A$, Lemma \ref{lem:splittingoffilling} implies that the (hyperbolic) quotient $\bar{G}$ has a $1$--edge splitting over $\bar{H}$.  Also, $\bar{H}$ is quasiconvex, and has height smaller than that of $H$.  Induction implies that $\bar{G}$ is virtually torsion-free.  Let $\bar{G}_0\dotsub\bar{G}$ be torsion-free, and let $G_0 = \phi^{-1}(\bar{G}_0)$.  The group $G_0$ is a torsion-free finite index subgroup of $G$.
\end{proof}

\subsection{Proof of Theorem \ref{t:QVH}}
We now prove Theorem \ref{t:QVH}.

\begin{proof}
We have already noted in Example \ref{ex:Hierarchy for cube} that if $G$ is hyperbolic and special then it admits a particular nice hierarchy and it is easy to see that this means any hyperbolic virtually special group is in $\QVH$.  Thus, it is the converse that we must prove -- any hyperbolic group in $\QVH$ is virtually special.

Certainly we know that $\{1 \}$ is (virtually) special.  Also, if $H$ is virtually special and $H \dotsub G$ then $G$ is clearly also virtually special.  Theorefore, it suffices to prove that if $G$ admits a one-edge splitting
with virtually special vertex group(s) and quasi-convex edge group $H$ then $G$ is virtually special.

We induct on the height of $H$ in $G$.  If the height is $0$, then $H$ is finite.  Proposition \ref{p:qvh is vtf} implies that there is some $G_0\dotnorm G$ which is torsion free.  But then $G_0$ is a graph of virtually special groups with trivial edge groups.  
In particular $G_0$ has a malnormal quasiconvex hierarchy terminating in virtually special groups.  By Theorem \ref{MQH}, $G_0$ (and hence $G$) is virtually special.

Now suppose the height of $H$ is positive.  The following is the key claim:
\begin{claim*}
  $H$ is separable in $G$.
\end{claim*}
Suppose the claim has been proved.  
By Proposition \ref{p:qvh is vtf}, $G$ is virtually torsion-free.
By Proposition \ref{p:qc+sepimpliesvm}, separable quasiconvex subgroups of $G$ are virtually malnormal.  Therefore, there is a finite-index (torsion-free) normal subgroup $G_0 < G$ so that  for any $G$-conjugate $H'$ of $H$, the subgroup $H' \cap G_0$ is malnormal in $G_0$.
The induced graph of groups for the hyperbolic group $G_0$ has malnormal and quasiconvex edge groups, and so $G_0$ (and hence $G$) is virtually special by Theorem \ref{MQH}.

\begin{proof}[Proof of claim]
  This argument follows the same structure as that for Proposition \ref{p:qvh is vtf}, as we induct on the height of $H$.  If $H$ is finite (height $0$) then we have already shown that $G$ is virtually special, so in particular $H$ is separable.  

  Now suppose that the 
  height of $H$ is $k>0$.  We wish to show that $H$ is separable, so choose $g\notin H$.
  We again consider the malnormal core $\mc{D}$ of $H$ and the induced peripheral structure $\mc{P}$ on $G$, and choose a peripherally finite filling $\phi\co G\to \bar{G} = G(N_1,\ldots,N_m)$ whose induced fillings on the vertex groups are virtually special (by Corollary \ref{msqt}), and so that furthermore:
  \begin{enumerate}
  \item $\phi$ restricts to the induced filling on $H$ and on the vertex groups (\ref{qcfilling}.\eqref{inducedfilling});
  \item $\bar{G}$ is hyperbolic and $\bar{H} = \phi(H)$ is quasiconvex in $\bar{G}$ (\ref{qcfilling}.\eqref{frq});
  \item the height of $\bar{H}$ in $\bar{G}$ is strictly less than $k$ (Theorem \ref{t:strictheightreduce}); and
  \item $\phi(g)\notin \bar{H}$ (\ref{qcfilling}.\eqref{gnotinH} with $F = \{g\}$).
  \end{enumerate}
As before (again using the results of Section \ref{sec:inducedsplittings} and in particular Lemma \ref{lem:splittingoffilling}), the quotient $\bar{G}$ has a 1--edge splitting over $\bar{H}$, and $\bar{H}$ is quasiconvex and has smaller height than $H$.  The vertex groups are virtually special.  By induction there is a finite quotient $\pi\co \bar{G}\to Q$ so that $\pi(\phi(g))\notin \pi(\bar{H})$.  But then $\pi\circ \phi$ gives a map to a finite quotient separating $g$ from $H$.
\end{proof}

\end{proof}

\section{Acknowledgments}
We thank the referee for helpful comments and corrections.

\small
\def\cprime{$'$}
\providecommand{\bysame}{\leavevmode\hbox to3em{\hrulefill}\thinspace}
\providecommand{\MR}{\relax\ifhmode\unskip\space\fi MR }
\providecommand{\MRhref}[2]{%
  \href{http://www.ams.org/mathscinet-getitem?mr=#1}{#2}
}
\providecommand{\href}[2]{#2}

\end{document}